\documentclass[a4paper,11pt,reqno]{article}

\usepackage[utf8]{inputenc}
\usepackage[T1]{fontenc}
\usepackage{lmodern}
\usepackage[english]{babel}
\usepackage{microtype}

\usepackage{amsmath,amssymb,amsfonts,amsthm}
\usepackage{mathtools,accents}
\usepackage{mathrsfs}
\usepackage{aliascnt}
\usepackage{braket}
\usepackage{bm}

\usepackage[a4paper,margin=3cm]{geometry}
\usepackage[citecolor=blue,colorlinks]{hyperref}

\usepackage{enumerate}
\usepackage{xcolor}

\makeatletter
\g@addto@macro\@floatboxreset\centering
\makeatother

\makeatletter
\def\newaliasedtheorem#1[#2]#3{
  \newaliascnt{#1@alt}{#2}
  \newtheorem{#1}[#1@alt]{#3}
  \expandafter\newcommand\csname #1@altname\endcsname{#3}
}
\makeatother

\numberwithin{equation}{section}

\newtheoremstyle{slanted}{\topsep}{\topsep}{\slshape}{}{\bfseries}{.}{.5em}{}

\theoremstyle{plain}
\newtheorem{theorem}{Theorem}[section]
\newaliasedtheorem{proposition}[theorem]{Proposition}
\newaliasedtheorem{lemma}[theorem]{Lemma}
\newaliasedtheorem{corollary}[theorem]{Corollary}
\newaliasedtheorem{counterexample}[theorem]{Counterexample}

\theoremstyle{definition}
\newaliasedtheorem{definition}[theorem]{Definition}
\newaliasedtheorem{question}[theorem]{Question}
\newaliasedtheorem{example}[theorem]{Example}
\newaliasedtheorem{conjecture}[theorem]{Conjecture}

\theoremstyle{remark}
\newaliasedtheorem{remark}[theorem]{Remark}

\newcommand{\setR}{\mathbb{R}}

\let\altphi\phi
\let\phi\varphi
\let\varphi\altphi
\let\altphi\undefined

\newcommand{\di}{\mathop{}\!\mathrm{d}}

\DeclareMathOperator{\supp}{supp}

\newcommand{\Ch}{{\sf Ch}}

\DeclareMathOperator{\Lip}{Lip}
\DeclareMathOperator{\Lipb}{Lip_b}

\newcommand{\leb}{\mathscr{L}}

\newcommand{\dist}{\mathsf{d}}

\newcommand{\meas}{\mathfrak{m}}

\DeclareMathOperator{\RCD}{RCD}
\newfont{\tmpf}{cmsy10 scaled 2500}

\newcommand{\mean}[1]{\,-\hskip-1.08em\int_{#1}}

\begin{document}
\title{Local spectral convergence in $\RCD^*(K,N)$ spaces}
\author{Luigi Ambrosio
\thanks{Scuola Normale Superiore, \url{luigi.ambrosio@sns.it}} \and
Shouhei Honda
\thanks{Tohoku University, \url{shonda@m.tohoku.ac.jp}}
} \maketitle

\begin{abstract} In this note we give necessary and sufficient conditions for the validity of the local spectral convergence, in balls, 
on the $\RCD^*$-setting.
\end{abstract}

\tableofcontents

\section{Introduction}
Let us consider a measured Gromov-Hausdorff (mGH, for short) convergent sequence of compact $n$-dimensional Riemannian manifolds $M_i$ to a compact metric measure space $(X, \dist, \meas)$, called a compact Ricci limit space,
\begin{equation}\label{eq:smooth}
\left( M_i, \frac{\mathrm{vol}}{\mathrm{vol}\,M_i} \right) \stackrel{mGH}{\to} (X, \dist, \meas),
\end{equation}
with the uniform Ricci curvature bound from below $\mathrm{Ric}_{M_i} \ge K >-\infty$.
Then the spectral convergence proven in \cite{CheegerColding3} by Cheeger-Colding gives
\begin{equation}\label{eq:spectral cc}
\lim_{i \to \infty}\lambda_k(X_i)=\lambda_k(X)
\end{equation}
for any $k$, where $\lambda_k$ denotes the $k$-th eigenvalue of Laplacian.
Moreover, any sequence of $k$-th eigenfunctions on $M_i$ with a uniform bound on $L^2$-norms has a 
uniform convergent subsequence to a $k$-th eigenfunction on $X$. 

This result was conjectured by Fukaya in \cite{Fukaya}, who introduced the notion of mGH-convergence and proved the same conlusion under bounded sectional curvature. This spectral convergence result has more recently been generalized to metric measure spaces with (Riemannian) Ricci curvature bounded from below, the so called $\RCD(K, \infty)/\RCD^*(K, N)$-metric measure spaces, by Gigli-Mondino-Savar\'e in \cite{GigliMondinoSavare13}. 
See Section 2 for a quick introduction to this class of metric measure spaces.

Next we discuss a noncompact case, i.e. a pointed mGH-convergent sequence of $n$-dimensional Riemannian manifolds $N_i$;
\begin{equation}\label{eq:noncompact}
\left(N_i, p_i, \frac{\mathrm{vol}}{\mathrm{vol}\,B_1(p_i)}\right) \stackrel{mGH}{\to} (Y, \dist, p, \nu)
\end{equation}
with $\mathrm{Ric}_{N_i} \ge K>-\infty$.

In several papers (e.g. \cite{Ding}, \cite{KuwaeShioya03}, \cite{Xu}, \cite{ZhangZhu16}), the local spectral convergence is investigated, i.e.
\begin{equation}\label{eq:local spectral convergence}
\lim_{i \to \infty}\lambda^D_k(B_R(p_i))=\lambda_k^D(B_R(p))
\end{equation}
for any $R>0$, together with the convergence of the corresponding eigenfunctions,
where $\lambda_k^D$ denotes the $k$-th eigenvalue of Laplacian associated with the Dirichlet problem on the open ball $B_R$. 
To be precise, we recall that for any $f\in H^{1, 2}_0(B_R(p), \dist, \meas)$ (the Sobolev closure of Lipschitz
 functions with compact support in $B_R(p)$), 
we say that $f$ is an eigenfunction with the eigenvalue $\lambda$ for the Dirichlet problem on $B_R(p)$ if  
\begin{equation}\label{eq:9}
\int_{B_R(p)}\Gamma(f,g) \dist \meas = \lambda \int_{B_R(p)}fg\dist \meas
\qquad\forall g \in H^{1, 2}_0(B_R(p), \dist, \meas),
\end{equation}
where $\Gamma$ is the {\it carr\'e du champ} operator associated to the metric measure structure. In
\cite{Ding} the question whether the limit function of $k$-eigenvalues still satisfies the Dirichlet boundary condition
is raised, but the main result of that paper, namely the convergence of the heat kernels, is proved independently of
an answer to this question. In \cite{KuwaeShioya03}, the spectral convergence as well as the Mosco
convergence of the local Cheeger energies with Dirichlet boundary conditions are claimed. Finally, the stability
of the Dirichlet boundary condition seems to play a role in Proposition~7.5 of \cite{Xu}.

One of the main purposes of the paper is to give an example such that (\ref{eq:local spectral convergence}) is not satisfied
in general,
providing at the same time positive results and, in particular,  convergence results for generic balls.

\begin{example}\label{ex:1}
We consider a trivial pointed mGH convergent sequence; 
$$
([0, +\infty), \dist_{eucl}, s, \mathcal{L}^1) \stackrel{mGH}{\to} ([0, +\infty), \dist_{eucl}, \pi/4, \mathcal{L}^1)
$$
as $s \uparrow \pi/4$.
It is easy to check that each $([0, +\infty), \dist_{eucl}, s, \mathcal{L}^1)$ is the pointed mGH limit spaces of a sequence of Riemannian metrics on $\mathbb{R}^2$ with nonnegative sectional curvature. In particular this sequence consists of $\RCD^*(0, 2)$-spaces.

For any $\epsilon \in (0, 1)$ let 
$$
f_{\epsilon}(t):=\cos \left( \frac{\pi t}{\pi-2\epsilon}\right).
$$
Since $f_{\epsilon}$ is smooth with $f_{\epsilon}(\pi/2-\epsilon)=0$, we have
$$
f_{\epsilon}|_{[0, \pi/2-\epsilon)} \in H^1_0(B_{\pi/4}(\pi/4-\epsilon), \dist_{eucl}, \mathcal{L}^1)
\left(=H^1_0([0, \pi/2-\epsilon), \dist_{eucl}, \mathcal{L}^1)\right).
$$
Moreover, for any $g \in C^{\infty}([0, \pi/2-\epsilon])$ with $g(\pi/2-\epsilon)=0$, integration by parts with $f_{\epsilon}'(0)=0$ yield
\begin{align*}
\int_0^{\pi/2-\epsilon}f_{\epsilon}'(t) g'(t)\dist t=\left[ f_{\epsilon}'g\right]_{0}^{\pi/2-\epsilon}-
\int_0^{\pi/2-\epsilon}f_{\epsilon}''(t) g(t) \dist t= -\int_0^{\pi/2-\epsilon}f_{\epsilon}''(t) g(t) \dist t,
\end{align*}
which shows that $f_{\epsilon}$ is an eigenfunction associated with the Dirichlet problem on the ball
$B_{\pi/4}(\pi/4-\epsilon)$, which coincides with $[0, \pi/2-\epsilon)$.
Note that the eigenvalue of $f_{\epsilon}$ is equal to $\pi^2/(\pi-2\epsilon)^2$, so that
$$
\lambda_1^D(B_{\pi/4}(\pi/4-\epsilon)) \le \pi^2/(\pi-2\epsilon)^2.
$$
On the other hand, it is well-known that $\lambda_1^D(B_{\pi/4}(\pi/4))=\lambda_1^D((0, \pi/2))=4$.
Thus 
\begin{equation}\label{eq:2}
\lim_{\epsilon \to 0^+}\lambda_1^D(B_{\pi/4}(\pi/4-\epsilon)) =1<4= \lambda_1^D(B_{\pi/4}(\pi/4)).
\end{equation}
\end{example}

Note that this is not yet a complete counterexample for the validity of (\ref{eq:local spectral convergence}), because the
approximating sequence is not smooth, but we will see later on how a diagonal argument can be used to complete
the counterexample.

The reason for the nonvalidity of (\ref{eq:local spectral convergence}) is that, in general, the limit $f$ of eigenfunctions $f_i \in H^{1, 2}_0(B_R(p_i), \dist_i, \frac{\mathrm{vol}}{\mathrm{vol}\,B_1(p_i)})$ is not in $H^{1, 2}_0(B_R(p), \dist, \nu)$ (but the identity (\ref{eq:9}) holds, as 
$f \in H^{1, 2}(B_R(p), \dist, \nu)$ and $g \in H^{1, 2}_0(B_R(p), \dist, \nu)$).
Indeed, in the case of Example~\ref{ex:1}, the limit $f(t)=\cos t$ of $f_{\epsilon}$ as $\epsilon \downarrow 0$ is not in 
the space $H^{1, 2}_0((0, \pi/2),\dist_{eucl}, \mathcal{L}^1)$.

The main result of the paper is to give a necessary and sufficient condition for the validity of the local spectral convergence on $\RCD^*$-spaces, where the meaning of local spectral convergence is taken in a little bit strong sense, with converging radii and centers 
(Definition \ref{def:local mosco}, Proposition~\ref{prop:equiv}). The condition can be stated as follows 
(Theorem~\ref{thm:equiv} with Lemma~\ref{lem:ch}):  for a $\RCD^*(K, N)$-space $(X, \dist, \meas)$ and $B_R(x)\subset X$,
\begin{equation}\label{eq:10}
H^{1, 2}_0(B_R(x), \dist, \meas )=\bigcap_{\epsilon>0}H^{1, 2}_0(B_{R+\epsilon}(x), \dist, \meas )
\end{equation}
if and only if the local spectral convergence on $B_R(x)$ holds for some/all mGH-convergent sequence of $\RCD^*(\hat{K}, \hat{N})$-spaces to $(X, \dist, \meas)$. We also prove in Lemma~\ref{lem:ch} that, for a given center $x$, \eqref{eq:10} holds with at most countably many exceptions.

In particular, in Example \ref{ex:1}, since 
$$
H^{1, 2}_0(B_{\pi/4}(s), \dist_{eucl}, \mathcal{L}^1) = \bigcap_{\epsilon>0}H^{1, 2}_0(B_{\pi/4+\epsilon}(s), \dist_{eucl}, \mathcal{L}^1)
$$
for any $s \in [0, \pi/4)$, by (\ref{eq:2}) and a diagonal argument we can find
a sequence of Riemannian metrics $g_i$ on $\mathbb{R}^2$ with nonnegative sectional curvature, $p_i \in \mathbb{R}^2$ and $k \in \mathbb{N}$ such that $$(\mathbb{R}^2, \dist_{g_i}, p_i, \frac{\mathrm{vol}}{\mathrm{vol}\, B_1(p_i)}) \stackrel{mGH}{\to} ([0, +\infty),  \dist_{eucl}, \pi/4, \frac{\mathcal{L}^1}{\mathcal{L}^1(B_1(p)))})$$
and $\lim_i\lambda_k^D(B_{\pi/4}(p_i)) \neq \lambda_k^D(B_{\pi/4}(\pi/4))$, which gives a 
full counterexample to the validity of (\ref{eq:local spectral convergence}).

Next we describe an application of our local spectral convergence result to the study of harmonic functions on $\RCD^*(K, N)$-spaces.
For that let us recall the following result ($\star$) of Petrunin given in \cite{Petrunin2} ;
\begin{enumerate}
\item[($\star$)] Let $(A_i^n, \dist_i, a_i)$ be a pointed Gromov-Hausdorff convergent sequence 
of $n$-dimensional Alexandrov spaces to a pointed noncollapsed 
Alexandrov space $(A^n, \dist, a)$ with a uniform lower bound on sectional curvature (then it is known that $(A_i^n, a_i, \dist_i, \mathcal{H}^n) \stackrel{mGH}{\to} (A^n, a, \dist, \mathcal{H}^n)$), let $f_i$ be harmonic functions on $B_R(a_i)$ 
(i.e. $f_i \in H^{1, 2}(B_R(a_i), \dist_i, \mathcal{H}^n)$ and (\ref{eq:9}) is satisfied with $\lambda=0$), 
and let $f$ be the locally uniform limit function of $f_i$ on $B_R(a)$. 
Then $f$ is harmonic on $B_r(a)$ for all $r\in (0,R)$ and
$$
\lim_{i \to \infty}\int_{B_r(a_i)}\Gamma_i(f_i)\dist \mathcal{H}^n \to \int_{B_r(a)}\Gamma(f)\dist \mathcal{H}^n
\qquad\forall r\in (0,R).
$$
\end{enumerate}
Moreover, in the same paper (see question 2.3 therein), Petrunin raised the following question: 
what happens in the result ($\star$) when the sequence of spaces is collapsed?

We can give a positive answer to the question in the $\RCD^*$-setting; $(X_i, \dist_i, x_i, \meas_i) \stackrel{mGH}{\to} (X, \dist, x, \meas)$ (see \cite{Honda2} for the corresponding results in the Ricci limit setting).
Then, combining this fact with the compatibility between Alexandrov and $\RCD$-spaces given in \cite{Petrunin, ZhangZhu}, we provide
an affirmative answer, as a particular case, in the Alexandrov setting (Theorem~\ref{thm:stability lap} and Corollary~\ref{cor:stab harm}). Moreover we will prove a kind
of  converse \textit{harmonic approximation} property: for any harmonic function $g$ on $B_R(x)$ and any $r \in (0, R)$, after passing to a subsequence, $g\vert_{B_r(x)}$ can be approximated by harmonic functions $g_i$ on $B_r(x_i)$, which is new even in the noncompact Ricci limit setting (see \cite{Honda3} for the corresponding results in the compact Ricci limit setting). 
Furthermore, we provide an example showing that the assumption, ``$r \in (0, R)$'', is needed (Remark~\ref{rem:example}).
Thus, these results provide a fairly complete picture of the stability and approximability of harmonic functions with respect to 
the mGH-convergence of $\RCD^*(K, N)$-spaces.

Let us introduce a key notion to prove the harmonic approximation, which is the \textit{harmonic replacement $\hat{f}$} of a function $f \in H^{1, 2}(B_R(x), \dist, \meas)$, i.e. $\hat{f}$ is uniquely determined by satisfying that $\hat{f}$ is a harmonic function on $B_R(x)$ with $f-\hat{f} \in H^{1, 2}_0(B_R(x), \dist, \meas)$ if $\lambda_1^D(B_R(x))>0$.
It is worth pointing out that this replacement naturally appeared in proofs of splitting theorems \cite{CheegerGromoll, CheegerColding} (see also Remark \ref{rem:excess}).
Our local spectral convergence shows; under mild assumptions, if $f_i$ converge strongly to $f$ on $B_R(x)$ in $H^{1, 2}$, then the harmonic replacements $\hat{f}_i$ is also $H^{1, 2}$-strongly convergent to $\hat{f}$ (Theorem \ref{cor:harm harm}), which plays a key role in the proof of the harmonic approximation

Finally, let us mention that these observations will be justified in more general setting of Poisson's equation $\Delta f=g$. This
provides new estimates for distance functions independent of almost nonnegative curvature assumptions 
(Remark~\ref{rem:excess} and Theorem~\ref{thm:metric cone distance}), which seem to be new even in the Ricci limit setting.

The paper is organized as follows. In Section~\ref{sec:1} we introduce the basic terminology and properties of $\RCD$-spaces,
passing then to the description of the basic stability results from \cite{GigliMondinoSavare13} and \cite{AmbrosioHonda}.
Then, for $U\subset X$ open, we introduce the local Sobolev spaces $H^{1,2}_0(U,\dist,\meas)$, $H^{1,2}(U,\dist,\meas)$ and the
related Laplacian operators. In Section~\ref{sec:2} we provide all basic stability results for problems with homogeneous
Dirichlet conditions, identifying a necessary and sufficient condition for the convergence. The final section of the paper
deals with elliptic problems in $H^{1,2}(B_R(x),\dist,\meas)$-Sobolev spaces, with possibly nonhomogeneous 
Dirichlet boundary conditions, and provides stability results also in this setting.

\smallskip
{\bf Acknowledgement.} The first author acknowledges the support of the MIUR PRIN2015 project ``Calculus of Variations''. 
The second author acknowledges the support of the JSPS Program for Advancing 
Strategic International Networks to Accelerate the Circulation of Talented Researchers, the
Grant-in-Aid for Young Scientists (B) 16K17585 and the warm hospitality of SNS.

\section{Notation and preliminary results}\label{sec:1}

We use the notation $B_r(x)$ for open balls and $\overline{B}_r(x)$ for $\{y:\ \dist(x,y)\leq r\}$.
We also use the standard notation $\mathrm{LIP}(X,\dist)$, $\mathrm{LIP}_b(X, \dist)$, $\mathrm{LIP}_c(X,\dist)$ for the
spaces of Lipschitz, bounded Lipschitz, compactly supported Lipschitz functions, respectively. 

The following elementary lemma will play a role in the proof of Lemma~\ref{lem:ch}, about the generic
coincidence of two classes of Sobolev spaces.

\begin{lemma} \label{lem:ulam} Let $(Z,\tau)$ be a topological space with a countable basis of open sets
and let $f_R,\,g_R$, $R>0$, be upper semicontinuous
functions from $Z$ to $\mathbb{R}$ satisfying $f_R\leq g_R\leq f_{R+\epsilon}$ for all $R,\,\epsilon>0$. 
Then $f_R\equiv g_R$ in $X$ for all $R>0$, with at most countably many exceptions.
\end{lemma}
\begin{proof} For any open set $A$, the functions 
$\tilde{f}(R):=\sup_A f_R$,  $\tilde{g}(R):=\sup_A g_R$
satisfy $\tilde{f}(R)\leq\tilde{g}(R)\leq\tilde{f}(R+\epsilon)$ for all $\epsilon>0$, 
hence the set $\{\tilde{f}\neq\tilde{g}\}$ is at most countable in $(0,+\infty)$. 
Choosing a countable basis $\{A_k\}$ of open sets of $(Z,\tau)$, since upper semicontinuity gives
$$
f_R(u)=\inf_{A_k\ni u} \sup_{A_k} f_R,\qquad
g_R(u)=\inf_{A_k\ni u} \sup_{A_k} g_R,
$$
the result follows.
\end{proof}

Let us now recall basic facts about Sobolev spaces and heat flow in metric measure spaces 
$(X,\dist,\meas)$, see \cite{AmbrosioGigliSavare13} and \cite{Gigli1} for a more systematic treatment of this topic. 
We shall always assume that the metric space $(X,\dist)$ is complete and separable. 

\begin{definition}\label{def:Chee}
The Cheeger energy
$\Ch=\Ch_{\dist,\meas}:L^2(X,\meas)\to [0,+\infty]$ is a convex and $L^2(X,\meas)$-lower semicontinuous functional defined as follows:
\begin{equation}\label{eq:defchp}
\Ch(f):=\inf\left\{\liminf_{n\to\infty}\frac 12\int_X|\nabla  f_n|^2\di\meas:\ \text{$f_n\in\Lipb(X,\dist)\cap L^2(X,\meas)$, $\|f_n-f\|_{L^2}\to 0$}\right\}, 
\end{equation}
where $|\nabla f|$ is the so-called slope, or local Lipschitz constant.

The Sobolev space $H^{1,2}(X,\dist,\meas)$ then concides with $\{f:\ \Ch(f)<+\infty\}$. 
\end{definition}

When endowed with the norm
$$
\|f\|_{H^{1,2}}:=\left(\|f\|_{L^2(X,\meas)}^2+2\Ch(f)\right)^{1/2}
$$
the space $H^{1,2}(X,\dist,\meas)$ 
is Banach, reflexive if $(X,\dist)$ is doubling (see \cite{AmbrosioColomboDiMarino}), and  
separable Hilbert if $\Ch$ is a quadratic form (see \cite{AmbrosioGigliSavare14}). 
According to the terminology introduced in \cite{Gigli1}, we say that a 
metric measure space $(X,\dist,\meas)$ is infinitesimally Hilbertian if $\Ch$ is a quadratic form.
  
By looking at minimal relaxed slopes and by a polarization procedure, one can then define a {\it carr\'e du champ}
$$
\Gamma:H^{1,2}(X,\dist,\meas)\times H^{1,2}(X,\dist,\meas)\rightarrow L^1(X,\meas)
$$
playing in this abstract theory the role of the scalar product between gradients (more precisely,
the duality between differentials and gradients, see \cite{Gigli1}). In infinitesimally Hilbertian metric measure
spaces the $\Gamma$ operator
satisfies all natural symmetry, bilinearity, locality and chain rule properties, and provides integral representation to
$\Ch$: $2\Ch(f)=\int_X \Gamma(f,f)\,\dist\meas$ for all $f\in H^{1,2}(X,\dist,\meas)$. We also adopt the usual
abbreviation
$$
\Gamma(f):=\Gamma(f,f).
$$

We can now define a densely
defined operator $\Delta:D(\Delta)\to L^2(X,\meas)$ whose domain consists of all functions $f\in H^{1,2}(X,\dist,\meas)$
satisfying
$$
 \int_X hg\dist\meas=-\int_X \Gamma(f,h)\dist\meas\quad\qquad\forall h\in H^{1,2}(X,\dist,\meas)
$$
for some $g\in L^2(X,\meas)$. The unique $g$ with this property is then denoted $\Delta f$.

Another object canonically associated to the metric measure structure, more specifically to $\Ch$, is the heat flow $h_t$, defined as the
$L^2(X,\meas)$ gradient flow of $\Ch$; even in general metric measure structures one can use the Brezis-Komura theory of 
gradient flows of lower semicontinuous functionals in Hilbert spaces to provide existence and uniqueness of this gradient flow. 
In the special case of infinitesimally Hilbertian 
metric measure spaces, this provides a linear, continuous and self-adjoint contraction
semigroup $h_t$ in $L^2(X,\meas)$, with the Markov property, characterized by:
$t\mapsto h_t f$ is locally absolutely continuous in $(0,+\infty)$ with values in $L^2(X,\meas)$ and
$$
\frac{d}{dt}h_t f=\Delta h_t f\quad\text{for $\leb^1$-a.e. $t\in (0,+\infty)$, for all $f\in L^2(X,\meas)$.}
$$

%

In order to introduce the class of $\RCD(K,\infty)$ and $\RCD^*(K,N)$ metric measure spaces we follow the
$\Gamma$-calculus point of view, based on Bochner's inequality, because this is the point of view more relevant
in our proofs. However, the equivalence with the Lagrangian point
of view, based on the theory of optimal transport first proved in \cite{AmbrosioGigliSavare15} (in the case $N=\infty$) and then in
\cite{ErbarKuwadaSturm}, \cite{AmbrosioMondinoSavare} (in the case $N<\infty$), 
plays indeed a key role in the proof of the results we need, mainly taken from \cite{GigliMondinoSavare13} 
and \cite{AmbrosioHonda}. 

\begin{definition} [$\RCD$ spaces]\label{def:RCDspaces} Let $(X,\dist,\meas)$ be a metric measure space, with $(X,\dist)$ complete, length,
satisfying
\begin{equation}\label{eq:Grygorian}
\meas\bigl(B_r(\bar x)\bigr)\leq c_1 e^{c_2r^2}\qquad\forall r>0
\end{equation}
for some $c_1,\,c_2>0$ and $\bar x\in X$ and the so-called Sobolev to Lipschitz property:  any  
$f\in H^{1,2}(X,\dist,\meas)\cap L^\infty(X,\meas)$ with $\Gamma(f)\leq 1$ $\meas$-a.e. in $X$ 
has a representative in $\tilde{f}\in\Lipb(X,\dist)$, with $\Lip(\tilde f)\leq 1$.

For $K\in\setR$, 
we say that $(X,\dist,\meas)$ is a $\RCD(K,\infty)$ metric measure space if, 
for all $f\in H^{1,2}(X,\dist,\meas)\cap D(\Delta)$ with $\Delta f\in H^{1,2}(X,\dist,\meas)$,
 Bochner's inequality
$$
\frac 12\Delta\Gamma(f)\geq \Gamma(f,\Delta f)+K\Gamma(f) 
$$
holds in the weak form
$$
\frac 12\int \Gamma(f)\Delta\phi\dist\meas\geq
\int\phi(\Gamma(f,\Delta f)+K\Gamma(f))\dist\meas 
\quad\forall \phi\in D(\Delta)\,\,\text{with $\phi\geq 0$,}\,\,\Delta\phi\in L^\infty(X,\meas).
$$

Analogously, for $K\in\setR$ and $N>0$, we say that $(X,\dist,\meas)$ is a $\RCD^*(K,N)$ metric 
measure space if, for all $f\in H^{1,2}(X,\dist,\meas)\cap D(\Delta)$ with $\Delta f\in H^{1,2}(X,\dist,\meas)$, 
Bochner's inequality
$$
\frac 12\Delta\Gamma(f)\geq \Gamma(f,\Delta f)+\frac 1N(\Delta f)^2+K\Gamma(f) 
$$
holds in the weak form
$$
\frac 12\int \Gamma(f)\Delta\phi\dist\meas\geq
\int\phi(\Gamma(f,\Delta f)+\frac 1N(\Delta f)^2+K\Gamma(f))\dist\meas 
$$
for all $\phi\in D(\Delta)$ with $\phi\geq 0$ and $\Delta\phi\in L^\infty(X,\meas)$.
\end{definition}

Since we are going to adopt the so-called extrinsic viewpoint in mGH convergence,  it will be convenient for us
not to add the assumption that $X=\supp\meas$, made in some other papers on this subject.
However, it is obvious that $(X,\dist,\meas)$ is $\RCD(K,\infty)$ (resp. $\RCD^*(K,N)$) if and only if
$(X,\dist,\supp\meas)$ is $\RCD(K,\infty)$ (resp. $\RCD^*(K,N)$).

Finally recall the existence of good cut-off functions which will play a key role in the next sections: 
for any $x \in X$ and all $0<r_1<R_1<+\infty$,  there exist $\phi \in\mathcal{D}(\Delta)$ such that
\begin{equation}\label{eq:good cf}
 0 \le \phi \le 1, \quad\phi|_{B_{r_1}(x)}\equiv 1, \quad\supp \phi \Subset B_{R_1}(x), \quad \Gamma (\phi ) + |\Delta \phi| \le C(K, N, r_1, R_1).
\end{equation}
See \cite[Lemma 3.1]{MondinoNaber} for the proof.
Note that this is a generalization of \cite[Theorem 6.33]{CheegerColding} from the smooth setting to the $\RCD$-setting.

\subsection{mGH convergence and global stability results}

From now on, $K\in\mathbb{R}$ and $N \in [1,+\infty)$. Let us fix a pointed measured Gromov-Hausdorff (mGH for short in the sequel) convergent sequence $(X_i, \dist_i, x_i, \meas_i) \stackrel{mGH}{\to} (X, \dist, x, \meas)$ of $\RCD^*(K, N)$-spaces. We adopt throughout the
paper the so-called
extrinsic approach, assuming that $X_i=\supp\meas_i$, $X=\supp\meas$ and that all the sets $X_i$, as well as $X$, are contained in a common
proper metric space $(Y,\dist)$, with $\dist_i\vert_{X_i\times X_i}=\dist$ and $x_i\to x$. 

Notice also that
the extrinsic approach is convenient to formulate various notions of convergence and to avoid the use of
$\epsilon$-isometries. However, it should be handled with care: for instance, if $f\in\mathrm{LIP}_b(Y,\dist)$ and we view
this as a sequence of bounded Lipschitz functions in the spaces $X_i$, then the sequence need not be strongly
convergent in $H^{1,2}$ (see \cite{AmbrosioStraTrevisan} for an example).

Notice also that the properness assumption on
$(Y,\dist)$ is justified by the uniform doubling property granted by the Bishop-Gromov inequality. Unlike $X_i$, the ambient space
$(Y,\dist)$ will not appear often in our notation, since the measures $\meas_i$ are concentrated on $X_i$; however $Y$
plays an important role to define weak convergence of functions $f_i\in L^p(X_i,\meas_i)$, since the test functions 
are continuous and compactly supported in the ambient space. Occasionally, when we want to emphasize the role of $Y$ (as in
the proof of Theorem~\ref{thm:compact loc sob}) we adopt the superscript $Y$. Notice also that any continuous (compactly
supported) function $\varphi:B_R^{X_i}(x)\to\mathbb R$ can be thought as the restriction of a continuous (compactly supported)
function $\tilde\varphi:B_R^Y(x)\to\mathbb R$.

In this setting, let us recall the definition of $L^2$-strong/weak convergence of functions with respect to the mGH-convergence.
The following formulation is due to \cite{GigliMondinoSavare13} and \cite{AmbrosioStraTrevisan}, which fits the pmG-convergence well.
Other good formulations of $L^2$-convergence, in connection with mGH-convergence, can be found in \cite{Honda2, KuwaeShioya03}.
However, in our setting these formulations are equivalent by the volume doubling condition (e.g. \cite[Proposition 3.3]{Honda7}). 

\begin{definition}[$L^2$-convergence of functions with respect to variable measures]\label{def:l2}
We say that $f_i \in L^2(X_i, \meas_i)$ \textit{$L^2$-weakly converge to $f \in L^2(X, \meas )$} 
if $\sup_i\|f_i\|_{L^2}<\infty$ and $f_i\meas_i \stackrel{C_c(Y)}{\rightharpoonup} f\meas$.
Moreover, we say that $f_i\in L^2(X_i,\meas_i)$ \textit{$L^2$-strongly converge to $f\in L^2(X,\meas)$} if 
$f_i$ $L^2$-weakly converge to $f$ with $\limsup_i\|f_i\|_{L^2}\le \|f\|_{L^2}$. 
\end{definition}

For nonnegative functions $f_i\in L^1(X_i,\meas_i)$, we also say that $f_i$ 
\textit{$L^1$-strongly converge to $f\in L^1(X,\meas)$} if $\sqrt{f_i}$ $L^2$ strongly
converge to $\sqrt{f}$.

Note that it was proven in \cite{GigliMondinoSavare13} (see also \cite{AmbrosioStraTrevisan}, \cite{AmbrosioHonda}) 
that any $L^2$-bounded sequence has an $L^2$-weak convergent subsequence in the sense above. In the sequel
we shall denote by $\Ch^i=\Ch_{\meas_i}$, $\Gamma_i$, $\Delta_i$, etc. the various objects associated to the $i$-th
metric measure structure.

The following is a consequence of the main results of \cite{GigliMondinoSavare13}, see also \cite{AmbrosioHonda}.

\begin{theorem}[Mosco convergence]\label{thm:Mosco}
The Cheeger energies $\Ch^i$ Mosco converge to $\Ch$, i.e. the 
following conditions hold:
\begin{itemize}
\item[(a)] (\emph{Weak-$\liminf$}). For every $f_i\in L^2(X_i,\meas_i)$ $L^2$-weakly converging to $f\in L^2(X,\meas)$, one has
\[ \Ch(f)\le \liminf_{i\to\infty} \Ch^i(f_i).\]
Moreover if $\sup_i\Ch^i(f_i)<\infty$, then $f_i$ $L^2$-strongly converge to $f$.
\item[(b)] (\emph{Strong-$\limsup$}). For every $f \in L^2(X,\meas)$ there exist $f_i\in L^2(X_i,\meas_i)$, $L^2$-strongly converging to $f$ with
\begin{equation}\label{eq:optimal_Chee} \Ch(f)=\lim_{i\to \infty} \Ch^i(f_i).\end{equation}
Moreover, if $f \in \mathrm{LIP}(X, \dist) \cap H^{1, 2}(X, \dist, \meas)$, then we can find $f_i$ in such a way that
$f_i \in \mathrm{LIP}(X_i, \dist_i) \cap H^{1, 2}(X_i, \dist_i, \meas_i)$ with $\sup_i\|\Gamma_i(f_i)\|_{L^{\infty}}<\infty$.
\end{itemize}
\end{theorem}

Next, we define in a natural way, following \cite{GigliMondinoSavare13}, weak and strong convergence in the Sobolev space $H^{1,2}$,
with a variable reference measure.

\begin{definition}[Convergence in the Sobolev spaces]
We say that $f_i\in H^{1,2}(X_i,\dist_i,\meas_i)$ are weakly convergent in $H^{1,2}$ to 
$f\in H^{1,2}(X_i,\dist_i,\meas_i)$ if $f_i$ are $L^2$-weakly convergent to $f$ 
and $\sup_i\Ch^i(f_i)$ is finite. Strong convergence
in $H^{1,2}$ is defined by requiring $L^2$-strong convergence of the functions, and that $\Ch(f)=\lim_i\Ch^i(f_i)$. 
\end{definition}

The following results are proved in \cite{AmbrosioHonda} (see Corollary~5.5, Theorem~5.7 and Lemma~5.8 therein), 
building on \cite{AmbrosioStraTrevisan};  the third result can be considered as the local counterpart
of  Theorem~\ref{thm:Mosco}.

\begin{theorem}[Weak stability of Laplacian]\label{thm:lap}
Let $f_i\in D(\Delta_i)$ with $$\sup_i (\|f_i\|_{L^2(X_i,\meas_i)}+\|\Delta_i f_i\|_{L^2(X_i,\meas_i)})<\infty$$
and assume that $f_i$ $L^2$-strongly converge to $f$. Then $f \in D(\Delta )$ and
\begin{itemize}
\item[(1)] $f_i$ strongly converge to $f$ in $H^{1, 2}$;
\item[(2)]   $\Delta_i f_i$ $L^2$-weakly converge to $\Delta f$.
\end{itemize}
\end{theorem}

\begin{theorem} [Continuity of the gradient operators] \label{thm:cont_reco} 
Let $v\in H^{1,2}(X,\dist,\meas)$ and let $v_i\in H^{1,2}(X_i,\dist_i,\meas_i)$
be strongly convergent in $H^{1,2}$ to $v$. Then:
\begin{itemize} 
\item[(i)] If $w_i$ weakly converge to $w$ in $H^{1,2}$ and $\Gamma_i(v_i,w_i)$ is uniformly bounded
in $L^p$ for some $p\in (1,\infty)$, then $\Gamma_i(v_i, w_i)$ $L^p$-weakly converge to $\Gamma(v,w)$.
\item[(ii)] If $w_i$ strongly converge to $w$ in $H^{1,2}$, then 
$\Gamma_i(v_i, w_i)$ $L^1$-strongly converge to $\Gamma(v,w)$.
\end{itemize} 
\end{theorem}

\begin{lemma}[Lower semicontinuity]\label{lem:sci}
Let $f \in H^{1,2}(X,\dist,\meas)$ and let $f_i\in H^{1,2}(X_i,\dist_i,\meas_i)$
be strongly convergent in $H^{1,2}$ to $f$. Then
$$
\liminf_{i\to\infty}\int_{X_i} g\sqrt{\Gamma_i(f_i)}\dist\meas_i\geq\int_Xg\sqrt{\Gamma(f)}\dist\meas
$$
for any lower semicontinuous function $g:Y\to [0,+\infty]$.
\end{lemma}

We conclude this section with the following technical lemma, which is an easy consequence of the previous
two stability results.

\begin{lemma}\label{lem:exist app}
For any $f \in \mathrm{LIP}_c(B_R(x), \dist)$ there exist $f_i \in \mathrm{LIP}_c(B_R(x_i), \dist_i)$ satisfying 
$\sup_i\|\Gamma_i (f_i)\|_{L^{\infty}}<\infty$ and strongly convergent to $f$ in $H^{1, 2}$. 
\end{lemma}
\begin{proof}
By Theorem~\ref{thm:Mosco} there exist $g_i \in \mathrm{LIP}(X_i, \dist_i) \cap H^{1, 2}(X_i, \dist_i, \meas_i)$ with 
$\sup_i\|\Gamma_i (g_i)\|_{L^{\infty}}<\infty$, strongly convergent to $f$ in $H^{1, 2}$.
Let $\epsilon >0$ with $\supp f \subset B_{R-3\epsilon}(x)$, set $r_1:=R-3\epsilon$, $R_1:=R-2\epsilon$ and let
$\phi_{i} \in \mathcal{D}(\Delta_i)$ be satisfying \eqref{eq:good cf}.
By Theorems \ref{thm:Mosco}, \ref{thm:lap} and the compactness with respect to the $L^2$-weak convergence, 
with no loss of generality we can assume that $\phi_{i}$ strongly converge to some $\phi \in \mathcal{D}(\Delta)$ in $H^{1, 2}$.
Then, since $\sup_i\|\Gamma_i (\phi_i)\|_{L^{\infty}}<\infty$, it is not difficult to check that $f_{i}:=\phi_ig_{i}$ strongly converge 
to $\phi f=f$ in $H^{1, 2}$, which completes the proof.
\end{proof}

\subsection{Local Sobolev spaces}

Next,  let us discuss local analysis on $\RCD^*(K, N)$-spaces, which is the main purpose of the paper.

\begin{definition}[Sobolev spaces $H^{1,2}_0$, $\hat{H}^{1,2}_0$]\label{def:loc sob}
Let $U$ be an open subset of $X$.
\begin{enumerate}
\item{($H^{1,2}_0$-Sobolev space)} We denote by $H^{1,2}_0(U, \dist, \meas )$ the 
$H^{1,2}$-closure of $\mathrm{LIP}_c(U, \dist)$.
\item{($\hat{H}^{1,2}_0$-Sobolev space)} We denote by $\hat{H}^{1,2}(U, \dist, \meas)$ the set of all
$f \in H^{1,2}(X,\dist,\meas)$ such that $f=0$ $\meas$-a.e. in $X\setminus U$. 
\end{enumerate}
\end{definition}

It is trivial that $H^{1,2}_0(U,\dist,\meas)$ and $\hat{H}^{1,2}_0(U, \dist, \meas)$ are closed subspaces of $H^{1,2}(X, \dist, \meas)$,
with $H^{1,2}_0(U,\dist,\meas)\subset\hat{H}^{1,2}_0(U, \dist, \meas)$. A kind of ``reverse'' inclusion is provided in the
following lemma.

\begin{lemma}\label{lem:ch} For all $x\in X$ and $R>0$ one has
$$\hat{H}^{1,2}_0(B_R(x), \dist, \meas)=\bigcap_{\epsilon>0}H^{1,2}_0(B_{R+\epsilon}(x), \dist, \meas).$$ 
In addition, for all $x\in X$, the equality $\hat{H}^{1,2}_0(B_R(x), \dist, \meas)=H^{1,2}_0(B_R(x), \dist, \meas)$
holds with at most countably many exceptions.
\end{lemma}
\begin{proof} The inclusion $\supset$ is a direct consequence of the definition and of the fact that
boundaries of balls are $\meas$-negligible (this follows by the thin annulus property \eqref{eq:annulus} ensured by the
doubling and length assumption). Let us check the converse inclusion.
Let $f \in \hat{H}^{1,2}_0(B_R(x), \dist, \meas)$ and let $\epsilon>0$.
Take $f_i \in \mathrm{LIP}(X, \dist) \cap H^{1, 2}(X, \dist, \meas)$ with $\|f_i -f\|_{H^{1, 2}} \to 0$ as $i \to \infty$, and take  $\phi \in \mathrm{LIP}_c(X, \dist)$ with $\phi|_{B_{R+\epsilon/2}(x)}\equiv 1$ and $\supp \phi \subset B_{R+\epsilon}(x)$.
Then, since $\phi f_i \in H^{1,2}_0(B_{R+\epsilon}(x), \dist, \meas)$ and $\phi f_i \to \phi f=f$ in $H^{1, 2}(X, \dist, \meas)$, 
we have $f \in H^{1,2}_0(B_{R+\epsilon}(x), \dist, \meas)$, which completes the proof of the first statement.

To prove the second one, it suffices to apply Lemma~\ref{lem:ulam} to the characteristic funtions of the
sets $H^{1,2}_0(B_R(x),\dist,\meas)$, $\hat{H}^{1,2}_0(B_R(x),\dist,\meas)$, defined in the separable Hilbert space 
$Z=H^{1,2}(X,\dist,\meas)$.
\end{proof}

Accordingly, for $U\subset X$ open we can define the local Cheeger-energy $\Ch_U:L^2(X, \meas) \to [0, +\infty]$ by 
\begin{equation}\label{eq:local energy}
\Ch_U(f):=
\begin{cases}\Ch(f) &\text{if $f\in H^{1, 2}_0(U, \dist, \meas)$;}\\
+\infty &\text{otherwise}
\end{cases}
\end{equation}
and put $\Ch_{(x,R)}:=\Ch_{B_R(x)}$.

Let us now discuss the Dirichlet problem on a ball $B_R(x)$.

\begin{definition}[Dirichlet Laplacian on balls]\label{def:dlapballs}
Let $\mathcal{D}_0(\Delta, B_R(x))$ denote the set of $f \in H^{1, 2}_0(B_R(x), \dist, \meas)$ such that 
there exists $h:=\Delta_{x,R} f\in L^2(B_R(x),\meas)$ satisfying
$$\int_{B_R(x)} hg\dist\meas=-\int_{B_R(x)} \Gamma(f,g)\dist\meas\qquad\forall g\in H^{1,2}_0(B_R(x),\dist,\meas).
$$
\end{definition}

Strictly speaking, this Dirichlet Laplacian $\Delta_{x,R}$ should not be confused with the operator $\Delta$, for this reason we adopted a
distinguished symbol.

It follows from standard arguments that the spectrum of $\Delta_{x,R}$ is discrete and unbounded
(except when $X$ consists of a single point), so we denote it by 
$$
0\le \lambda_1^D(B_R(x)) \le \lambda^D_2(B_R(x)) \le \cdots \to +\infty,
$$
counted with multiplicities.
Moreover, if $X \setminus B_{R+\epsilon}(x) \neq \emptyset$ for some $\epsilon>0$, then $\lambda_1^D(B_R(x))>0$, which is a direct consequence of Sobolev inequalities (c.f. (\ref{eq:00})).

Let us now introduce the local Sobolev space $H^{1, 2}(B_R(x), \dist, \meas)$ on a $\RCD^*(K, N)$-space $(X, \dist, \meas)$. See for instance \cite{Cheeger,Shanmugalingam} for the definition of general local Sobolev space $H^{1, p}(U, \dist, \meas)$ for any $p \in [1, \infty)$ and any open subset $U$ of $X$, see also Remark~\ref{rem:compare}. Our working definition is this:

\begin{definition}[$H^{1,2}(U,\dist,\meas)$-Sobolev space]\label{def:reddu}
Let $U\subset X$ be open. We say that $f\in L^2(U,\meas)$ belongs to $H^{1,2}(U,\dist,\meas)$ if:
\begin{itemize} 
\item[(i)] $\phi f \in H^{1, 2}(X, \dist, \meas)$ for any $\phi \in \mathrm{LIP}_c(U, \dist)$;
\item[(ii)] $\Gamma(f)\in L^1(U,\meas)$. 
\end{itemize}
\end{definition}

Notice that condition (i) corresponds precisely to a $H^{1,2}_{\rm loc}$ property, namely (i)
holds if and only if  for any $V\Subset U$ there exists $\tilde f\in H^{1,2}(X,\dist,\meas)$ with $\tilde f\equiv f$ on $V$.
Condition (ii) makes sense, since the locality properties of $\Gamma$ ensure that $\Gamma(f)$ makes sense for all 
functions $f$ as in (i). Indeed, choosing
a sequence of functions $\chi_n\in \mathrm{LIP}_c(U, \dist)$ with $\{\chi_n=1\}\uparrow U$ and defining
$$
\Gamma(f):=\Gamma(f\chi_n)\qquad\text{$\meas$-a.e. on $\{\chi_{n+1}=1\}\setminus\{\chi_n=1\}$}
$$
we obtain an extension of the carr\'e du champ operator on $H^{1,2}(U,\dist,\meas)$
(for which we keep the same notation, being also independent of the choice of $\chi_n$) 
which retains all bilinearity and locality properties.  

{\rm \begin{remark}\label{rem:compare}
Even though it does not play a role in this paper, it is worth to compare the space $H^{1,2}(U,\dist,\meas)$ with the space
$H^{1,2}(\overline{U},\dist,\meas_{\overline{U}})$ (i.e. we apply Definition~\ref{def:Chee} to the space $\overline{U}$ with the
induced distance and measure $\meas_{\overline{U}}$) and with the space $W^{1,2}$ of Cheeger's paper \cite{Cheeger}: 
a function $f\in L^2(U)$ is said to belong to $W^{1,2}(U,\dist,\meas)$ if there exist $f_n\in L^2(U,\meas)$ convergent to $f$ 
in $L^2$ and upper gradients $g_n$ of $f_n$ with $\sup_n\|g_n\|_{L^2(U)}<+\infty$. 

Summing up, we have
$$
H^{1,2}(\overline{U},\dist,\meas_{\overline{U}})\subsetneq H^{1,2}(U,\dist,\meas)=W^{1,2}(U,\dist,\meas).
$$
and the strict inclusion may occur even when $\meas(\partial U)=0$. Since, as we said, this does not
play a role in the results of our paper, we only outline the arguments (notice that neither doubling nor
Poincar\'e are involved here, only local compactness is needed).

\smallskip
\noindent
{\bf $H^{1,2}(\overline{U},\dist,\meas_{\overline{U}})\subset H^{1,2}(U,\dist,\meas)$.} Property (i) is obvious,
since any function $f\in H^{1,2}(\overline{U},\dist,\meas_{\overline{U}})$ can be approximated
in $L^2(\overline{U},\meas_{\overline{U}})$ by bounded Lipschitz functions $f_n$ with 
$$
\sup_{n}\int_{\overline{U}}|\nabla f_n|^2\dist\meas<+\infty.
$$
One can then apply the lower semicontinuity on open sets to $f_n\chi$, with $\chi\in \mathrm{LIP}_c(U, \dist)$, to obtain that
$$
\int_V \Gamma(f)\leq\liminf_{n\to\infty}\int_V\Gamma(f_n\chi)\dist\meas\leq
\sup_{n}\int_{\overline{U}}|\nabla f_n|^2\dist\meas
$$
for any open set $V$ with $V\Subset\{\chi=1\}$. By monotone convergence, one then gets
that $\Gamma(f)\in L^1(U,\meas)$.

\smallskip
\noindent
{\bf $H^{1,2}(U,\dist,\meas)\subset W^{1,2}(U,\dist,\meas)$.} For all $f\in H^{1,2}(U,\dist,\meas)$, exploiting
property (i) of Definition~\ref{def:reddu} and arguing as in the proof of the classical Meyers-Serrin theorem,
one can prove the existence of locally Lipschitz functions $f_n:U\to\mathbb R$ convergent to $f$ in $L^2(U,\meas)$ with 
$\limsup_n\||\nabla f_n|\|_{L^2(U)}\leq\|\Gamma(f)\|_{L^1(U)}$. Since the slope is an upper gradient for 
locally Lipschitz functions, this proves the inclusion.

\smallskip
\noindent
{\bf $W^{1,2}(U,\dist,\meas)\subset H^{1,2}(U,\dist,\meas)$.} This inclusion requires the identification between weak
upper gradients and relaxed gradients estabilished in full generality in \cite{AmbrosioGigliSavare13}. Indeed, if $f\in W^{1,2}(U,\dist,\meas)$
and if $g$ is any weak limit point in the $L^2(U)$ topology of upper gradients $g_n$ of $f_n$, $f_n\to f$ in $L^2(U,\meas)$,
then we know that $g$ is 2-weak upper gradient, according to the theory developed in \cite{Shanmugalingam}. Then, for any open
set $V\Subset U$ we can apply the identification of \cite{AmbrosioGigliSavare13} to get $h_n\in\mathrm{LIP}(\overline{V},\dist)$ with $h_n\to f$
in $L^2(\overline{V},\meas)$ and $\limsup_n\||\nabla h_n|\|_{L^2(\overline{V})}\leq \|g\|_{L^2(\overline{V})}$.
Since $V$ is arbitrary, this immediately gives that $\phi f\in H^{1,2}(X,\dist,\meas)$ for all $\phi\in\mathrm{LIP}_c(U,\dist)$,
and that $\sup_{V\Subset U}\|\Gamma(f)\|_{L^1(V)}<+\infty$.

\smallskip
\noindent
{\bf Counterexample.} An example of a function $f\in  H^{1,2}(U,\dist,\meas)\setminus H^{1,2}(\overline{U},\dist,\meas_{\overline{U}})$ can be found in Remark \ref{rem:example}.
\end{remark}

\begin{definition}[Laplacian on balls]
For $f \in H^{1, 2}(B_R(x), \dist, \meas)$, we write $f\in D(\Delta, B_R(x))$ if there exists 
$h:=\Delta_{x,R}f\in L^2(B_R(x),\meas)$  satisfying
$$
\int_{B_R(x)} hg\dist\meas=-\int_{B_R(x)} \Gamma(f, g) \dist\meas \qquad
\forall g\in H^{1,2}_0(B_R(x),\dist,\meas).
$$
\end{definition}

Since for $f \in H^{1, 2}_0(B_R(x), \dist, \meas)$ one has $f\in D(\Delta, B_R(x))$ iff
$f\in D_0(\Delta,B_R(x))$ and the laplacians are the same, we retain the same notation $\Delta_{x,R}$
of Definition~\ref{def:dlapballs}.
It is easy to check that for any $f \in \mathcal{D}(\Delta, B_R(x))$ and any 
$\phi \in \mathcal{D}(\Delta ) \cap \mathrm{LIP}_c(B_R(x), \dist)$ with $\Delta \phi \in L^{\infty}(X, \meas)$
one has (understanding $\phi\Delta_{x,R}f$ to be null out of $B_R(x)$)
$\phi f \in \mathcal{D}(\Delta )$ with
\begin{equation}\label{eq:local to global}
\Delta (\phi f)=f\Delta \phi +2\Gamma( \phi,f) +\phi \Delta_{x,R} f
\qquad\text{$\meas$-a.e. in $X$.}
\end{equation}

Finally, we recall the Sobolev inequality, which plays a role in the proof of stability of $H^{1,2}$-functions.
\begin{align}\label{eq:sobolev ineq}
\biggl(\mean{B_R(x)}|f-\mean{B_R(x)}f\dist\meas|^{2^*}\dist\meas\biggr)^{1/2^*}\leq C\biggl(\mean{B_R(x)}\Gamma(f)\dist\meas\biggr)^{1/2}
\quad\forall f \in H^{1, 2}(B_R(x), \dist, \meas),
\end{align}
where $2^*=2N/(N-2)$ if $N>2$, $2^*$ can be any power in $(2,\infty)$ if $N\in (1,2]$
and $C:=C(K, N, 2^*,R)>0$. In our context of curvature-dimension bounds, \eqref{eq:sobolev ineq} can be proved for
locally Lipschitz functions starting from the 
local Poincar\'e inequality of \cite{Rajala}, applying then Theorem~5.1 of \cite{HaKo} in combination with the
Bishop-Gromov inequality. By density, it extends to global $H^{1,2}$-functions. Since
$H^{1,2}(B_R(x),\dist,\meas)$ locally coincide with global $H^{1,2}$-functions, a simple
monotone approximation then provides the result in the class $H^{1,2}(B_R(x),\dist,\meas)$.

We also need the following volume estimate for ``thin'' annuli; for any $\epsilon >0$ there exists $\delta :=\delta (K, N, R, \epsilon)>0$ such that 
for all $x\in X$ one has
\begin{equation}\label{eq:annulus}
\meas (B_r(x) \setminus B_{(1-\delta)r}(x)) \le \epsilon \meas (B_r(x)) \qquad\forall r<R.
\end{equation}
See for instance \cite{ColdingMinicozzi} or \cite{Sturm06}.

\section{Stability of local problems with homogeneous Dirichlet conditions}\label{sec:2}

Let us start to discuss local stability on $\RCD^*$-spaces with respect to the mGH-convergence.

\begin{proposition}[Compactness of $\hat{H}_0^{1, 2}$-functions]\label{prop:compact}
Any sequence $(f_i)$ with $f_i \in \hat{H}_0^{1, 2}(B_R(x_i), \dist_i, \meas_i)$ and $\sup_i\|f_i\|_{H^{1, 2}}<\infty$ has a weak 
$H^{1, 2}$-convergent subsequence to some $f \in \hat{H}^{1, 2}_0(B_R(x), \dist, \meas)$.
\end{proposition}
\begin{proof}
By Theorems \ref{thm:Mosco}, \ref{thm:lap} and the compactness with respect to the $L^2$-weak convergence, 
with no loss of generality  we can assume that $f_i$ $H^{1, 2}$-weakly converge to some $f \in H^{1, 2}(X, \dist, \meas)$.
Thus it suffices to check that $f \in \hat{H}^{1, 2}_0(B_R(x), \dist, \meas)$.
Let $z \in X \setminus \overline{B}_R(x)$, let $r>0$ with $B_r(z) \subset  X \setminus \overline{B}_R(x)$ and let $\phi \in \mathrm{LIP}_c(X, \dist)$ with $\supp \phi \subset B_r(z)$.
Then, applying Lemma~\ref{lem:exist app} to $f:=\phi$ shows that there exist $\phi_i \in \mathrm{LIP}_c(X_i, \dist_i)$ 
strongly convergent to $\phi$ in $H^{1, 2}$ such that $\supp \phi_i\subset B_r(z_i)$, where $z_i \to z$.
In particular
$$
\int_X\phi f\dist \meas =\lim_{i \to \infty}\int_{X_i}\phi_if_i\dist \meas_i=0.
$$
Since $\phi$ is arbitrary, this proves that $f \in \hat{H}_0^{1, 2}(B_R(x), \dist, \meas)$.
\end{proof}

\begin{definition}[Mosco convergence of local Cheeger energies]\label{def:local mosco}
We say that the \textit{local Cheeger energies $\Ch_{\mathrm{loc}}^i$ Mosco converge to 
the local Cheeger energy $\Ch_{\mathrm{loc}}$ at $(z, R) \in X \times (0, +\infty)$}, denoted by 
\begin{equation}\label{eq:local mosco}
\Ch_{\mathrm{loc}}^i \to \Ch_{\mathrm{loc}}\,\,\mathrm{at}\,(z, R)
\end{equation}
for short, if whenever $R_i \to R$ and $z_i\ni X_i \to z$, the following conditions hold;
\begin{itemize}
\item[(a)] (\emph{Weak-$\liminf$}). For every $f_i\in L^2(B_{R_i}(z_i),\meas_i)$ $L^2$-weakly converging to $f\in L^2(B_R(z),\meas)$, one has
\[ \Ch_{(z, R)}(f)\le \liminf_{i\to\infty} \Ch^i_{(z_i, R_i)}(f_i).\]
\item[(b)] (\emph{Strong-$\limsup$}). For every $f \in L^2(B_R(z),\meas)$ there exist $f_i\in L^2(B_{R_i}(z_i),\meas_i)$, $L^2$-strongly converging to $f$ with
\begin{equation}\label{eq:optimal_locChee} \Ch_{(z, R)}(f)=\lim_{i\to \infty} \Ch^i_{(z_i, R_i)}(f_i).\end{equation}
\end{itemize}
\end{definition}

The next proposition follows from a standard argument. For the reader's convenience we give a sketch of the proof.

\begin{proposition}[Equivalence of Mosco and local spectral convergence]\label{prop:equiv}
The following properties are equivalent:
\begin{enumerate}
\item[(1)] $\Ch^i_{\mathrm{loc}} \to \Ch_{\mathrm{loc}}$ at $(z, R)$.
\item[(2)] The spectral convergence for the Dirichlet problems on $B_{R_i}(z_i)$ holds for any $R_i \to R$ and any $z_i \to z$, i.e. for any $k$, any sequence of eigenfunctions $f_{i, k}$ with the eigenvalue $\lambda_k^D(B_{R_i}(z_i))$ associated with the Dirichlet problem on $B_{R_i}(z_i)$ has a $L^2$-strong convergent subsequence to an eigenfunction $f_k \in \mathcal{D}_0(\Delta, B_R(z))$ with the eigenvalue $\lambda_k^D(B_R(z))$ (in particular $\lim_i\lambda_k^D(B_{R_i}(z_i))=\lambda_k^D(B_R(z))$ holds). 
\end{enumerate}
\end{proposition}
\begin{proof}
See for instance \cite[Theorem 7.8]{GigliMondinoSavare13} for the proof of the implication from (1) to (2).
We give only a proof of the converse implication.

Assume that (2) holds. 
Since Lemma~\ref{lem:exist app} yields the condition (b) in Definition~\ref{def:local mosco}, it suffices to check the condition (a).
Let $f_i \in H^{1, 2}_0(B_R(x_i), \dist_i, \meas_i)$ with $\sup_i\|f_i\|_{H^{1, 2}}<\infty$ and let $f \in L^2(B_R(x), \meas)$ be the $L^2$-weak limit function.
Then by Theorem \ref{thm:Mosco} $f$ is the $L^2$-strong limit function and $f \in H^{1, 2}(X, \dist, \meas)$.
By assumption, with no loss of generality we can assume that $\{f_{i, k}\}_k, \{f_k\}_k$ are o.n.b. in 
$H^{1, 2}_0(B_{R_i}(z_i), \dist_i, \meas_i)$, in $H^{1, 2}_0(B_R(z), \dist, \meas)$, respectively and that 
$f_{i, k}$ $L^2$-strongly converge to $f_k$ for any $k$. 

Let $f=\sum_{k=1}a_kf_k$ in $L^2(B_R(z), \meas)$, $f_i:=\sum_{k=1}^Na_{i, k}f_{i, k}$ in $H^{1, 2}_0(B_{R_i}(z_i), \dist_i, \meas_i)$, 
and denote by $f^N:=\sum_{k=1}^Na_kf_k$, $f_i^N:=\sum_{k=1}^Na_{i, k}f_{i, k}$ the corresponding finite sums.
Since 
$$a_{i, k}=\int_{B_{R_i}(z_i)}f_i f_{i, k}\dist \meas_i \to \int_{B_R(z)}f f_k\dist \meas=a_i$$
as $i \to \infty$, we have
\begin{align*}
\|f^N\|_{H^{1, 2}}^2&= \sum_{k=1}^N(a_k)^2\left(1 + \int_{B_R(z)} \Gamma (f_k) \dist \meas \right) \\
&=\sum_{k=1}^N(a_k)^2(1+\lambda_k^{D}(B_R(z))) \\
&=\lim_{i \to \infty}\sum_{k=1}^N(a_{i, k})^2(1+\lambda_k^{D}(B_R(z_i)))\\
&=\lim_{i \to \infty}\|f_i^N\|_{H^{1, 2}}^2 \le \liminf_{i \to \infty}\|f_i\|_{H^{1, 2}}^2<\infty.
\end{align*}
Thus letting $N \uparrow \infty$ shows (a).
\end{proof}

We are now in a position to introduce the main result of the paper (recall that, according to Lemma~\ref{lem:ch}, for given
$z$ condition (2) below holds for all radii, with at most countably many exceptions).

\begin{theorem}[Main equivalence]\label{thm:equiv}
The following are equivalent;
\begin{enumerate}
\item[(1)] $\Ch^i_{\mathrm{loc}} \to \Ch_{\mathrm{loc}}$ at $(z, R)$.
\item[(2)] $H^{1, 2}_0(B_R(z), \dist, \meas)=\hat{H}^{1, 2}_0(B_R(z), \dist, \meas)$.
\end{enumerate}
In particular if these conditions hold, then the Mosco convergence of local Cheeger energies at $(z, R)$ holds for any
sequence of $\RCD^*(K,N)$-spaces $(Y_i, \dist_i, \nu_i) \stackrel{GH}{\to} (X, \dist, \meas)$. 
\end{theorem}
\begin{proof}
We first prove the implication from (1) to (2). Assume that (1) holds.

Let $f \in \hat{H}^{1, 2}(B_R(z), \dist, \meas)$. Then by Lemma \ref{lem:ch} for any $\epsilon>0$ there exists $f_{\epsilon} \in \mathrm{LIP}_c(B_{R+\epsilon}(z), \dist)$ such that $\|f-f_{\epsilon}\|_{H^{1, 2}}<\epsilon$.
Applying Lemma \ref{lem:exist app} for $f:=f_{\epsilon}$ yields that there exist $f_{\epsilon, i} \in \mathrm{LIP}_c(B_{R+\epsilon}(z_i), \dist_i)$ such that $\sup_i\|\Gamma_i (f_{\epsilon, i})\|_{L^{\infty}}<\infty$ and $f_{\epsilon, i}$ strongly converge to $f_{\epsilon}$ in $H^{1, 2}$.
Thus there exists a subsequence $j(i)$ such that $g_{j(i)}:=f_{i^{-1}, j(i)}$ strongly converge to $f$ in $H^{1, 2}$.
Then the condition (a) in Definition \ref{def:local mosco} shows that $f \in H^{1, 2}_0(B_R(z), \dist, \meas)$, which proves (2).

Next assume that (2) holds. By the same reason as in the proof of Proposition~\ref{prop:equiv}, 
it suffices to check the condition (a) in Definition~\ref{def:local mosco}.
This is a direct consequence of Proposition~\ref{prop:compact} and assumption (2).
\end{proof}

\begin{remark} By scaling and translation invariance, Lemma~\ref{lem:ch} obviously gives
\begin{equation}\label{eq:example}
H^{1,2}_0(B_R(x), \dist_{eucl}, \mathcal{L}^n)=\hat{H}^{1, 2}_0(B_R(x), \dist_{eucl}, \mathcal{L}^n)
\qquad\forall (x, R) \in \mathbb{R}^n \times (0, +\infty).
\end{equation}
In particular any mGH-convergent sequence of $\RCD^*(K, N)$-spaces to $(\mathbb{R}^n, \dist_{eucl}, 0_n, \mathcal{H}^n)$ satisfies $\Ch^i_{\mathrm{loc}} \to \Ch_{\mathrm{loc}}$ at any $(x, R) \in \mathbb{R}^n \times (0, \infty)$.

Still using scaling invariance, one can prove an analogous property for all $R>0$ when
the metric measure space is a metric cone and $x$ is the pole. See Proposition \ref{prop:cone sobolev}.

\end{remark}

\begin{remark}[Correction to \cite{Honda}]
In \cite{Honda} the second author established second-order (or weak $C^{1, 1}$-) differential structure on a Ricci limit space $(X, \dist, \meas)$.
In the proof, one of the key estimates was
\begin{equation}\label{eq:bound}
\int_{B_r(x)}\Gamma(\Gamma(f))\dist \meas<\infty
\end{equation}
for any limit harmonic function $f: B_R(x) \to \mathbb{R}$ and any $r<R$, which means that $f$ is the uniform limit 
function of smooth harmonic functions $f_i:B_R(x_i) \to \mathbb{R}$.
In order to prove (\ref{eq:bound}), the local spectral convergence was used in his first proof.
However, as we observed in the introduction, this is not the correct argument. However, the problem can be fixed
by using the local spectral convergence for almost every radius granted by this paper. 
\end{remark}

\begin{remark}
It is natural to ask whether the Mosco convergence of the modifed Cheeger energies
$\hat{\Ch}^i_{B_R(x_i)}:L^2(X_i, \meas_i) \to [0, +\infty]$ holds or not, where
\begin{equation}\label{eq:local energy2}
\hat{\Ch}_{B_R(x_i)}^i(f):=
\begin{cases}\Ch(f) &\text{if $f\in \hat{H}^{1, 2}_0(B_R(x_i), \dist_i, \meas_i)$;}\\
+\infty &\text{otherwise.}
\end{cases}
\end{equation}
The answer is negative in general and we give an example as follows. Put $\hat{\Ch}_{(x_i,R)}:=\hat{\Ch}^i_{B_R(x_i)}$.

Let us consider the mGH-convergent sequence 
$$
(\mathbb{S}^1(s), \dist_{\mathbb{S}^1(s)}, \mathcal{H}^1) \stackrel{mGH}{\to} 
(\mathbb{S}^1(1), \dist_{\mathbb{S}^1(1)}, \mathcal{H}^1).
$$
As $s \downarrow 1$, let $x_s \in \mathbb{S}^1(s)$ with $x_s \to x_1 \in \mathbb{S}^1(1)$.
Note that 
$$
\hat{H}^{1, 2}_0(B_{\pi}(x_1), \dist_{\mathbb{S}^1(1)}, \mathcal{H}^1)=
H^{1, 2}(\mathbb{S}^1(1), \dist_{\mathbb{S}^1(1)}, \mathcal{H}^1),
$$
because $\mathcal{H}^1 (\mathbb{S}^1(1) \setminus B_{\pi}(x_1))=0$. In particular take $1 \in \hat{H}^{1, 2}_0(B_{\pi}(x_1), \dist_{\mathbb{S}^1(1)}, \mathcal{H}^1)$.
For any $s>1$ fix a canonical local isometry $\phi_s:(-\pi, \pi) \to B_{\pi}(x_s)$.
Let us prove that there is no approximating sequence $f_i \in \hat{H}^{1, 2}_0(B_{\pi}(x_{s_i}), \dist_{\mathbb{S}^1(s_i)}, \mathcal{H}^1)$ as $s_i \downarrow 1$ such that $f_i$ $L^2$-strongly converge to $1$ on $B_{\pi}(x)$ with $\sup_i\|f_i\|_{H^{1, 2}}<\infty$, which in particular implies that $\hat{\Ch}^i_{(x_{s_i}, \pi)}$ does not Mosco converge to $\hat{\Ch}_{(x_1, \pi)}$.
Indeed if such $f_i$ exist, then since $g_i:= f_i \circ \phi_{s_i} \in H^{1, 2}_0((-\pi, \pi), \dist_{eucl}, \mathcal{L}^1)$ $L^2$-strongly converge to $1$ with $\sup_i\|g_i\|_{H^{1, 2}}<\infty$, where we used the identity $H^{1, 2}_0=\hat{H}_0^{1, 2}$ on $B_{\pi}(x_{s_i})$, we would have $1 \in H^{1, 2}_0((-\pi, \pi), \dist_{eucl}, \mathcal{L}^1)$, which is a contradiction.

On the other hand by using the arguments in this section it is not difficult to see that the Mosco convergence of $\hat{\Ch}^i_{\mathrm{loc}}$-energies above is also equivalent to the spectral convergence, where the meaning of the spectrum here is taken by satisfying (\ref{eq:9}) in the $\hat{H}_0^{1, 2}$-setting (instead of $H^{1,2}_0$), and that if (2) of Theorem \ref{thm:equiv} is satisfied, then $\hat{\Ch}^i_{(x_i, R)}$ Mosco converge to $\hat{\Ch}_{(x, R)}=\Ch_{(x, R)}$.
In particular the Mosco convergence of $\hat{\Ch}^i_{\mathrm{loc}}$-energies is also satisfied under a generic assumption. 
\end{remark}

\begin{remark}[$p$-energies]
Even though we discussed only $H^{1, 2}$-Sobolev spaces for simplicity, it is easy to see that the $H^{1, p}$, $p>1$, version of all 
results above holds (except for Proposition \ref{prop:equiv}), i.e. using $p$-Cheeger energy instead of $\Ch$. See \cite{AmbrosioHonda}
for the Mosco convergence of the $p$-Cheeger energies.
\end{remark}

Let us denote by $h_t^{(x, R)}$ the heat flow on $L^2(B_R(x), \meas)$ associated with the Dirichlet problem on $B_R(x)$, namely
the gradient flow of $\Ch_{(x,R)}$ w.r.t. to the $L^2$ distance. Note that it follows from the Hilbertian theory of gradient flows that 
for any $f \in L^2(B_R(x), \meas)$ and $t>0$ one has, $h_t^{(x, R)}f \in H^{1, 2}_0(B_R(x), \dist, \meas) \cap \mathcal{D}(\Delta, B_R(x))$ with 
\begin{equation}\label{eq:bound on lap}
\|\Gamma( h_t^{(x, R)}f)\|_{L^1(B_R(x))}\le \frac{1}{t}\|f\|_{L^2(B_R(x))},\,\,\,\,
\|\Delta_{x,R} h_t^{(x, R)}f\|_{L^2(B_R(x))} \le \frac{1}{t}\|f\|_{L^2(B_R(x))}.
\end{equation}

\begin{proposition}\label{prop:conv heat flow}
Let  $f_i \in L^2(B_R(x_i), \meas_i)$ $L^2$-strongly converge to $f \in L^2(B_R(x), \meas)$ on $B_R(x)$. 
If
\begin{equation}\label{eq:generic}
H^{1, 2}_0(B_R(x), \dist, \meas)=\hat{H}^{1, 2}_0(B_R(x), \dist, \meas),
\end{equation}
then for any $t>0$ one has:
\begin{enumerate}
\item[(1)] $h_t^{(x_i, R)}f_i$ converge strongly to $h_t^{(x, R)}f$ in $H^{1, 2}$;
\item[(2)] $\Gamma_i (h_t^{(x_i, R)}f_i)$ $L^1$-strongly converge to $\Gamma(h_t^{(x, R)}f)$.
\end{enumerate}
\end{proposition}
\begin{proof} Since Theorem \ref{thm:equiv} and \eqref{eq:generic} ensure Mosco convergence of the local Cheeger energies, 
it follows from a standard argument that $h_t^{(x_i, R)}f_i$ $L^2$-strongly converge to $h_t^{(x, R)}f$ (c.f. \cite{GigliMondinoSavare13, KuwaeShioya03}).
On the other hand, taking (\ref{eq:bound on lap}) into account, we can pass to the limit as $i\to\infty$ in
$$
\int_{B_R(x_i)}\Gamma_i(g_i,h_t^{(x_i,R)}f_i)\dist\meas_i
=-\int_{B_R(x_i)}g_ih_t^{(x_i,R)}f_i\dist\meas_i\qquad g_i\in H^{1,2}_0(B_R(x_i),\dist_i,\meas_i)
$$
to obtain that $\Delta_{x_i,R} h_t^{(x_i, R)}f_i$ $L^2$-weakly converge to $\Delta_{x,R} h_t^{(x, R)}f$ on $B_R(x)$.
Thus 
\begin{align}\label{eq:oo}
\int_{X_i}\Gamma_i(h^{(x_i, R)}_tf_i)\dist \meas_i&=\int_{B_R(x_i)}\Gamma_i(h^{(x_i, R)}_tf_i)\dist \meas_i \nonumber \\
&=\int_{B_R(x_i)}h_t^{(x_i, R)}f_i\Delta_{x_i,R} h_t^{(x_i, R)}f_i\dist \meas_i \nonumber \\
&\to \int_{B_R(x)}h_t^{(x, R)}f\Delta_{x_i,R} h_t^{(x, R)}f\dist \meas \nonumber \\
&=\int_{B_R(x)}\Gamma(h_t^{(x, R)}f)\dist \meas=\int_{X}\Gamma(h_t^{(x, R)}f)\dist \meas,
\end{align}
which shows (1). Property (2) follows from (1) and Theorem \ref{thm:cont_reco}.
\end{proof}

\section{Harmonic replacements and nonhomogeneous boundary conditions}\label{sec:3}

From now on we consider a pointed mGH-convergent sequence of $\RCD^*(K, N)$-spaces, with $N>1$: 
$$
(X_i, x_i, \dist_i, \meas_i) \stackrel{mGH}{\to} (X, x, \dist, \meas).
$$

In this setting, the notions of weak/strong $L^p$ convergence on balls will be used in this final
section. These concepts derive immediately from the global ones by multiplying by characteristic functions, for instance
we say that $f_i\in L^2(X_i,\meas_i)$ $L^2$-weakly converge to $f\in L^2(X,\meas)$ on $B_R(x)$ if
$f_i1_{B_R(x_i)}$ $L^2$-weakly converge to $f1_{B_R(x)}$.

\begin{definition}[Local $H^{1, 2}$-convergence]
We say that $f_i \in H^{1, 2}(B_R(x_i), \dist_i, \meas_i)$ are weakly convergent in $H^{1, 2}$ to 
$f \in H^{1, 2}(B_R(x), \dist, \meas)$ on $B_R(x)$ if $f_i$ are $L^2$-weakly (or $L^2$-strongly, equivalently) 
convergent to $f$ on $B_R(x)$ with $\sup_i\|f_i\|_{H^{1, 2}}<\infty$.
Strong convergence in $H^{1, 2}$ on $B_R(x)$ is defined by requiring 
$\lim_i\|\Gamma_i(f_i)\|_{L^1(B_R(x_i))}=\|\Gamma(f)\|_{L^1(B_R(x))}$.
\end{definition}

\begin{theorem}[Compactness of local Sobolev functions]\label{thm:compact loc sob}
Let $R>0$ and let $f_i \in H^{1, 2}(B_{R}(x_i), \dist_i, \meas_i)$ with $\sup_i\|f_i\|_{H^{1, 2}}<\infty$.
Then there exist $f \in H^{1, 2}(B_{R}(x), \dist, \meas)$ and a subsequence $f_{i(j)}$ such that $f_{i(j)}$ $L^2$-strongly converge to $f$ on $B_R(x)$
and  
$$
\liminf_{j \to \infty}\int_{B_R(x_{i(j)})}\Gamma_{i(j)}(f_{i(j)})\dist \meas_{i(j)} \ge \int_{B_R(x)}\Gamma(f)\dist \meas.
$$  
\end{theorem}
\begin{proof}  For any $r<R$, fix $\phi^r\in \mathrm{LIP}_c(B^Y_R(x), \dist)$ with
$0 \le \phi^r\le 1$ and that $\phi^r\equiv 1$ on $B^Y_r(x)$.

We first check the $L^2$-strong compactness. Although it was proven in \cite{Honda2}, we give another proof here for 
the reader's convenience. Note that by (\ref{eq:sobolev ineq}) we have 
\begin{equation}\label{eq:equi bound}
\sup_i\|f_i\|_{L^{2^*}(B_R(x_i))}<\infty.
\end{equation}
Thus, by the $L^p$-weak compactness, with no loss of generality we can assume that $f_i$ $L^2$-weakly converge 
to $f \in L^{2^*}(B_R(x), \meas)$ on $B_R(x)$.

Since $\phi^rf_i$ $L^2$-weakly converge to $\phi^rf$ with $\sup_i\|\phi^rf\|_{H^{1, 2}}<\infty$, 
Theorem \ref{thm:Mosco} yields that $\phi^rf_i$ $L^2$-strongly converge to $\phi^rf$, 
and that $\phi^rf \in H^{1, 2}(X, \dist, \meas)$. Since $r<R$ is arbitrary, the second property easily implies
$$
\phi f \in H^{1, 2}(X, \dist, \meas)\qquad \forall \phi \in \mathrm{LIP}_c(B_R(x), \dist).
$$
On the other hand, the H\"older inequality shows 
$$
\|\phi^rf_i-f_i\|_{L^2(B_R(x_i))} \le \| 1_{B_R(x_i) \setminus B^Y_r(x)}f_i\|_{L^2(B_R(x_i))} \le 
\meas_i(B^Y_R(x_i) \setminus B^Y_r(x))^{1-2/2^*}\|f_i\|_{L^{2^*}(B_R(x_i))}^2.
$$
Since $\limsup_i\meas_i(B^Y_R(x_i) \setminus B^Y_r(x))\leq\meas(\overline{B}^Y_R(x) \setminus B^Y_r(x))$, 
combining this with (\ref{eq:annulus}) and (\ref{eq:equi bound}) yields 
\begin{equation}\label{eq:conv l2 norm}
\lim_{i \to \infty}\|f_i\|_{L^2(B_R(x_i))}=\|f\|_{L^2(B_R(x))},
\end{equation}
i.e.  $f_i$ $L^2$-strongly converge to $f$ on $B_R(x)$.

Assuming with no loss of generality that $\sqrt{\Gamma_i(f_i)}$ $L^2$-weakly converge to $g\in L^2(B_R(x),\meas)$ in $B_R(x)$
(i.e. the measures $1_{B_R(x_i)}\sqrt{\Gamma_i(f_i)}\meas_i$ weakly converge to
$g\meas$), if we 
prove that $g\geq\sqrt{\Gamma(f)}$, we are done. Fix an open set $A\Subset B_R^Y(x)$, then $s\in (0,R)$ with
$A\Subset B_s^Y(x)$ and $\phi\in \mathrm{LIP}_c(B_s^Y(x), \dist)$ with $0 \le \phi \le 1$ and
$\phi\vert_{B_s^Y(x)}\equiv 1$.
Since the sequence $f_i\phi$ converges to $f\phi$ weakly in $H^{1,2}$ on $X$, from Lemma~\ref{lem:sci} with
$g$ equal to the characteristic function of $A$ we get
$$
\int_{\overline{A}}g\dist\meas\geq\limsup_i\int_A\sqrt{\Gamma_i(f_i)}\dist\meas_i\geq\int_A\sqrt{\Gamma(f)}\dist\meas.
$$
Now, since any open set $B\subset B_R^Y(x)$ can be written as $\cup_i A_i$, with $A_i\Subset A_{i+1}$, the
previous inequality gives $\int_Bg\dist\meas\geq\int_B\sqrt{\Gamma(f)}\dist\meas$. Since $B$ is arbitrary,
this proves the inequality $g\geq\sqrt{\Gamma(f)}$.
\end{proof}

\begin{corollary}\label{cor:conv}
Let $f_i, \,g_i  \in H^{1, 2}(B_R(x_i), \dist_i, \meas_i)$ be $H^{1, 2}$-strong/weak convergent sequences to 
$f, \,g \in H^{1, 2}(B_R(x), \dist, \meas)$ on $B_R(x)$, respectively. Then 
\begin{equation}\label{eq:6}
\int_{B_R(x_i)}\Gamma_i( f_i, g_i) \dist \meas_i \to \int_{B_R(x)}\Gamma(f,g) \dist \meas.
\end{equation}
In particular, $f_i \pm g_i$ converge strongly to $f \pm g$ in $H^{1, 2}$ on $B_R(x)$ if $g_i$ is also $H^{1, 2}$-strong convergent on $B_R(x)$.
\end{corollary} 
\begin{proof}
It follows from $\liminf_i\|\Gamma (f_i +t g_i)\|_{L^1(B_R(x_i))} \ge \|\Gamma(f +t g)\|_{L^1(B_R(x))}$ for any $t \in \mathbb{R}$,
by polarization.
\end{proof}

\begin{theorem}[Stability of Laplacian on balls]\label{thm:stability lap}
Let $f_i \in \mathcal{D}(\Delta, B_R(x_i))$ with 
$$
\sup_i(\|f_i\|_{H^{1, 2}(B_R(x_i))}+\|\Delta f_i\|_{L^2(B_R(x))})<\infty,
$$
and let $f$ be the $L^2$-strong limit function of $f$ on $B_R(x)$ (so that, by Theorem~\ref{thm:compact loc sob}, 
$f \in H^{1, 2}(B_R(x), \dist, \meas)$). Then: 
\begin{enumerate}
\item[(1)] $f \in \mathcal{D}(\Delta, B_R(x))$;
\item[(2)] $\Delta_{x,R} f_i$ $L^2$-weakly converge to $\Delta_{x,R} f$ on $B_R(x)$; 
\item[(3)] $\Gamma_i (f_i)$ $L^1$-strongly converge to $\Gamma(f)$ on $B_r(x)$ 
for any $r<R$. 
\end{enumerate}
\end{theorem}
\begin{proof}
Fix $r<R$.
By using good cut-off functions (\ref{eq:good cf}) for pairs $B_r(x_i) \subset B_{r+(R-r)/2}(x_i)$, with no loss of generality we can assume 
the existence of a strong $H^{1, 2}$-convergent sequence of $\phi_i \in \mathrm{LIP}_c(B_R(x_i), \dist_i) \cap \mathcal{D}(\Delta)$ to $\phi \in \mathrm{LIP}_c(B_R(x), \dist) \cap \mathcal{D}(\Delta)$ with $0 \le \phi_i \le 1$, $|\Delta \phi_i| +\Gamma_i(\phi_i) \le C(K, N, r, R)$ and $\phi_i\vert_{B_r(x_i)}\equiv 1$, $0 \le \phi\le 1$, $|\Delta \phi| +\Gamma(\phi) \le C(K, N, r, R)$ and $\phi\vert_{B_r(x)}\equiv 1$.

Then, applying Theorem~\ref{thm:lap} to $\phi_if_i$ with the $L^2$-weak compactness (on $B_R(x)$) yields that (1) and (2) are satisfied, 
and that $\Gamma_i(\phi_if_i)$ $L^1$-strongly converge to $\Gamma(\phi f)$. This completes the proof of (3), 
because $\Gamma_i(\phi_if_i)=\Gamma_i(f_i)$ $\meas_i$-a.e. on $B_r(x_i)$ and $\Gamma (\phi f)=\Gamma(f)$ 
$\meas$-a.e. on $B_r(x)$. 
\end{proof}

The following is a direct consequence of Theorem~\ref{thm:stability lap}.

\begin{corollary}[Stability of harmonic functions]\label{cor:stab harm}
Let $f_i\in H^{1,2}(B_R(x_i),\dist_i,\meas_i)$ be harmonic functions on $B_R(x_i)$, i.e. $f_i \in \mathcal{D}(\Delta, B_R(x_i))$ 
with $\Delta_{x_i,R} f_i =0$. If $f_i$ $L^2$-strongly converge to $f$ on $B_R(x)$ with $\sup_i\|f_i\|_{H^{1, 2}}<\infty$, 
then $f$ is also harmonic on $B_R(x)$.
\end{corollary}

\begin{theorem}[Continuity of the local gradient operators]\label{thm:local gradient}
Let $f_i \in H^{1, 2}(B_R(x_i), \dist_i, \meas_i)$ be an $H^{1, 2}$-strong convergent sequence to $f \in H^{1, 2}(B_R(x), \dist, \meas)$ on $B_R(x)$.  
Then we have the following;
\begin{enumerate}
\item[(1)] $\Gamma_i(f_i)$ $L^1$-strongly converge to $\Gamma(f)$ on $B_R(x)$. 
\item[(2)] if $g_i \in H^{1, 2}(B_R(x_i), \dist_i, \meas_i)$ converge weakly to $g \in H^{1, 2}(B_R(x), \dist, \meas)$ in $H^{1, 2}$ on $B_R(x)$, and
if $\Gamma_i(f_i,g_i)$ is uniformly $L^p$-bounded for some $p \in (1, \infty)$, then $\Gamma_i(f_i,g_i)$
$L^p$-weakly converge to $\Gamma( f, g)$ on $B_R(x)$.
\end{enumerate}
\end{theorem}
\begin{proof}
Let us prove (1). As in the proof of Theorem~\ref{thm:compact loc sob}, 
we can assume with no loss of generality that $\sqrt{\Gamma_i(f_i)}$ $L^2$-weakly converge to
some function $g$ in $B_R(x)$, and $H^{1,2}$-strong convergence together with lower semicontinuity
provides, in this case, the inequality $\|g\|_{L^2(B_R(x))}\leq \|\sqrt{\Gamma(f)}\|_{L^2(B_R(x))}$. Since in the proof of 
Theorem~\ref{thm:compact loc sob}
we obtained, using only weak $H^{1,2}$ convergence, the inequality $g\geq\sqrt{\Gamma(f)}$, we obtain that
$g=\sqrt{\Gamma(f)}$, and this proves (1).

Statement (2) can be proved by applying Theorem~\ref{thm:cont_reco}(i) 
to the functions $g_i\phi^r$, $f_i\phi^r$, with $\phi^r$ as in the
proof of Theorem~\ref{thm:compact loc sob}. 
\end{proof}

From now on we discuss the local Dirichlet problem. The next lemma is quite standard.

\begin{lemma}\label{lem:existence}
Let $g \in L^2(B_R(x), \meas)$ and let $f \in H^{1, 2}(B_R(x), \dist, \meas)$.
If $\lambda_1^D(B_R(x))>0$, then there exists a unique $\hat{f} \in \mathcal{D}(\Delta, B_R(x))$ such that 
\begin{equation}
\begin{cases}
\Delta_{x,R}\hat{f}=g\\\\
f-\hat{f} \in H^{1, 2}_0(B_R(x), \dist, \meas).
\end{cases} 
\end{equation}
In addition
\begin{equation}
\label{eq:sol bound}
\|\sqrt{\Gamma(\hat{f})}\|_{L^2(B_R(x))}\leq 2\|\sqrt{\Gamma(f)}\|_{L^2(B_R(x))}+\frac{\|g\|_{L^2(B_R(x))}}{\lambda_1^D(B_R(x))},
\end{equation}
\begin{equation}\label{eq:l2 est}
\|\hat f\|_{L^2(B_R(x))}\leq \|f\|_{L^2(B_R(x))}+\frac{1}{\lambda_1^D(B_R(x))}\|\sqrt{\Gamma(f)}\|_{L^2(B_R(x))}+
\frac{\|g\|_{L^2(B_R(x))}}{(\lambda_1^D(B_R(x)))^2}.
\end{equation}
\end{lemma}
\begin{proof} Since 
$$
\lambda_1^D(B_R(x))\int_{B_R(x)}|h|^2\dist \meas \le \int_{B_R(x)}\Gamma(h)\dist \meas\qquad\forall h \in H^{1, 2}_0(B_R(x), \dist, \meas),
$$
it follows that $\sqrt{\int_{B_R(x)}\Gamma(\cdot ) \dist \meas}$ is an equivalent norm on $H^{1, 2}_0(B_R(x), \dist, \meas)$.
Let us consider the linear functional $\mathcal{F}$ on $H^{1, 2}_0(B_R(x), \dist, \meas)$ defined by
$$
\mathcal{F}(h):=-\int_{B_R(x)}(gh +\Gamma(f,h)) \dist \meas.
$$
Then since $\mathcal{F}$ is continuous with $\|\mathcal F\|\leq \|\sqrt{\Gamma(f)}\|_{L^2}+\|g\|_{L^2}/\lambda_1^D(B_R(x))$, 
there exists a unique $\phi \in H^{1, 2}_0(B_R(x), \dist, \meas)$ such that $\mathcal{F}(h)=\int_{B_R(x)}\Gamma( \phi,  h) \dist \meas$, i.e.
$$
-\int_{B_R(x)}(gh +\Gamma( f, h ))\dist \meas=\int_{B_R(x)}\Gamma( \phi,  h ) \dist \meas
\qquad \forall h \in H^{1, 2}_0(B_R(x), \dist, \meas).
$$
Letting $\hat{f}:=f+\phi$ and using the inequality $\|\sqrt{\Gamma(\phi)}\|_{L^2}\leq \|\sqrt{\Gamma(f)}\|_{L^2}+\|g\|_{L^2}/\lambda_1^D(B_R(x))$
completes the proof of the existence and the apriori estimates. Uniqueness is a simple consequence of linearity and apriori estimates.
\end{proof}

In the particular case when $g=0$, the function $\hat{f}$ provided by the previous lemma will be called
\textit{harmonic replacement} of $f$. 

\begin{theorem}[Continuity of the local Dirichlet problem]\label{thm:conv stab}
Let $f_i \in H^{1, 2}(B_R(x_i), \dist_i, \meas_i)$ be a weak $H^{1, 2}$-convergent sequence to $f \in H^{1, 2}(B_R(x), \dist, \meas)$ on $B_R(x)$, and let $g_i \in L^2(B_R(x_i), \meas_i)$ be an $L^2$-weak convergent sequence to $g \in L^2(B_R(x), \meas)$ on $B_R(x)$.
Assume that $\lambda_1^D(B_R(x))>0$ and that 
$$
H^{1, 2}_0(B_R(x), \dist, \meas)=\hat{H}^{1, 2}_0(B_R(x), \dist, \meas).
$$
Then the solutions $\hat{f}_i \in \mathcal{D}(\Delta_i, B_R(x_i))$ of the Dirichlet problems, $\Delta_{x_i,R} \hat{f}_i=g_i$ with $f_i-\hat{f}_i \in H^{1, 2}_0(B_R(x_i), \dist_i, \meas_i)$,  converge weakly in $H^{1,2}$ to the solution $\hat{f}$ of the Dirichlet problem, $\Delta_{x,R} \hat{f}=g$ with $f-\hat{f} \in H^{1, 2}_0(B_R(x), \dist, \meas)$.
Moreover the convergence is strong in $H^{1, 2}$ if and only if 
\begin{equation}\label{eq:equiv strong}
\int_{B_R(x_i)}\Gamma_i(\hat{f}_i,f_i) \dist \meas_i \to \int_{B_R(x)}\Gamma( \hat{f}, f) \dist \meas.
\end{equation}
Finally (\ref{eq:equiv strong}) is satisfied if either $\sup_i\|\Gamma( f_i)\|_{L^p(B_R(x_i))}<\infty$ for some $p>1$, or 
$f_i \in H^{1, 2}_0(B_R(x_i), \dist_i, \meas_i)$ for all $i$.
\end{theorem}
\begin{proof}
It is easy to check that (\ref{eq:sol bound}) and (\ref{eq:l2 est}) give $\sup_i\|\hat{f}_i\|_{H^{1, 2}}<\infty$.
Thus, by Theorem~\ref{thm:compact loc sob}, with no loss of generality we can assume that there exists 
the weak $H^{1, 2}$-limit function $h \in H^{1, 2}(B_R(x), \dist, \meas)$ of $\hat{f}_i$ on $B_R(x)$.
Moreover, Theorem~\ref{thm:equiv} yields 
\begin{equation}\label{eq:compact supp}
f-h \in H^{1, 2}_0(B_R(x), \dist, \meas),
\end{equation}
and Theorem~\ref{thm:stability lap} shows that $h \in \mathcal{D}(\Delta, B_R(x))$ with $\Delta_{x,R} h=g$.
In particular our assumption, $\lambda_1^D(B_R(x))>0$, together with the uniqueness part of Lemma~\ref{lem:existence} 
yield $f=h$, which completes the proof of the first part. 

Next we assume that (\ref{eq:equiv strong}) is satisfied.
Then since $\hat{f}_i-f_i$ $L^2$-strongly converge to $\hat{f}-f$ on $B_R(x)$, we have 
\begin{align*}
\|\Gamma_i(\hat{f}_i)\|_{L^1(B_R(x_i))}&=\int_{B_R(x_i)}(\Gamma_i( \hat{f}_i, f_i)  +\Gamma_i(\hat{f}_i,  (\hat{f}_i-f_i)) \dist \meas_i \\
&=\int_{B_R(x_i)}(\Gamma_i(\hat{f}_i, f_i ) -g_i (\hat{f}_i-f_i)) \dist \meas_i \\
&\to \int_{B_R(x)}(\Gamma(\hat{f}, f) -g(\hat{f}-f)) \dist \meas =\|\Gamma(\hat{f})\|_{L^1(B_R(x))},
\end{align*}
which yields that $\hat{f}_i$ converge strongly to $\hat{f}$ on $B_R(x)$ in $H^{1, 2}$.
The converse implication can be checked by a similar argument.

Finally we prove the sufficiency of the conditions mentioned in the statement, for the validity of (\ref{eq:equiv strong}).
If $f_i \in H^{1, 2}_0(B_R(x_i), \dist_i, \meas_i)$ for all $i$, then since Theorem \ref{thm:equiv} shows $f \in H^{1, 2}_0(B_R(x), \dist, \meas)$, we have
$$
\int_{B_R(x_i)}\Gamma_i( \hat{f}_i, f_i) \dist \meas_i=
-\int_{B_R(x_i)}g_if_i\dist \meas_i \to -\int_{B_R(x)}gf\dist \meas=\int_{B_R(x)}\Gamma( \hat{f}, f) \dist \meas.
$$

Next we assume that $\sup_i\|\Gamma_i(f_i)\|_{L^p(B_R(x_i))}<\infty$ for some $p>1$.
Then since Young's inequality shows
$$
|\Gamma_i(\hat{f}_i,f_i )|^q \le |\Gamma_i(\hat{f}_i)|^{q/2}|\Gamma_i(f_i)|^{q/2} \le 
\frac{q}{2}\Gamma_i(\hat{f}_i) + \frac{2-q}{2}\Gamma_i(f_i)^{q/(2-q)}
$$
for any $q \in (1, 2)$, applying this in the case when $p=2q/(2-q)$ yields $\sup_i\|\Gamma_i(\hat{f}_i,  f_i )\|_{L^q(B_R(x_i))}<\infty$. 
In particular Theorem~\ref{thm:local gradient} yields that $\Gamma( \hat{f},  f) \in L^q(B_R(x), \meas)$.

Note that Theorems \ref{thm:stability lap} and \ref{thm:local gradient} yield
\begin{equation}\label{eq:conti gradi}
\int_{B_r(x_i)}\Gamma_i(\hat{f}_i,  f_i) \dist \meas_i \to \int_{B_r(x)}\Gamma(\hat{f}, f)\dist \meas\qquad \forall r<R.
\end{equation}
Since 
$$
\left| \int_{B_R(x_i)}\Gamma_i(\hat{f}_i, f_i) \dist \meas_i-\int_{B_r(x_i)}\Gamma_i(\hat{f}_i, f_i)\dist \meas_i\right| \le 
\meas_i(B_R(x_i) \setminus B_r(x_i))^{1/\hat{q}} \|\Gamma_i(\hat{f}_i,  f_i)\|_{L^q(B_R(x_i))},
$$
where $\hat{q}^{-1}+q^{-1}=1$, (\ref{eq:annulus}) and (\ref{eq:conti gradi}) show (\ref{eq:equiv strong}).
\end{proof}

\begin{remark}\label{rem:example}
In Theorem~\ref{thm:conv stab} the assumption that $\lambda_1^D(B_{R_1}(x))>0$ is essential. We give an example as follows. 
Let us consider the mGH-convergent sequence 
$$
(\mathbb{S}^1(s), \dist_{\mathbb{S}^1(s)}, \mathcal{H}^1) \stackrel{mGH}{\to} 
(\mathbb{S}^1(1), \dist_{\mathbb{S}^1(1)}, \mathcal{H}^1).
$$
As $s \uparrow 1$, let $x_s \in \mathbb{S}^1(s)$ with $x_s \to x_1 \in \mathbb{S}^1(1)$ and let $r:=\pi$.
Note that $\lambda^D_1(B_{R_1}(x_1))=0$ for any $R_1>\pi$ because 
$$
\mathrm{LIP}_c(B_{R_1}(x_1), \dist_{\mathbb{S}^1(1)})=
\mathrm{LIP}(\mathbb{S}^1(1), \dist_{\mathbb{S}^1(1)}),
$$
hence the $H^{1,2}_0$ and the $H^{1,2}$ spaces coincide.

Take $f \in C^{\infty}(\mathbb{R})$ with $f(0) \neq f(2\pi)$ and fix a canonical local isometry $\phi: B_{\pi}(x_1) \to (0, 2\pi)$.
Then $g:=f\circ\phi$ belongs to $\mathcal{D}(\Delta, B_{\pi}(x))$ (with $\Delta_{x,\pi} f \in C^{\infty}(B_{\pi}(x_1))$ and bounded), but 
\begin{equation}\label{eq:nonextend}
g \neq h\vert_{B_{\pi}(x_1)}\qquad\forall h \in H^{1, 2}(\mathbb{S}^1(1), \dist_{\mathbb{S}^1(1)}, \mathcal{H}^1).
\end{equation}
Indeed, if $g=h\vert_{B_{\pi}(x_1)}$ for some $h \in H^{1, 2}(\mathbb{S}^1(1), \dist_{\mathbb{S}^1(1)}, \mathcal{H}^1)$, then since 
$\Gamma(g)\in L^{\infty}(B_{\pi}(x_1))$, we see that $h \in \mathrm{LIP}(\mathbb{S}^1(1), \dist_{\mathbb{S}^1(1)})$, which contradicts 
our assumption that $f(0) \neq f(2\pi)$.

Then there is no approximating sequence $f_i \in H^{1, 2}(B_{\pi}(x_i), \dist_{\mathbb{S}^1(s_i)}, \mathcal{H}^1)$ as $s_i \uparrow 1$ such that $f_i$ $L^2$-strongly converge to $g$ on $B_{\pi}(x)$ with $\sup_i\|f_i\|_{H^{1, 2}(B_{\pi}(x_i))}<\infty$.
Indeed, if such $f_i$ exist, then $B_{\pi}(x_i)=\mathbb{S}^1(s_i)$ and Theorem~\ref{thm:Mosco} yield 
$g \in H^{1, 2}(\mathbb{S}^1(1), \dist_{\mathbb{S}^1(1)}, \mathcal{H}^1)$, which contradicts (\ref{eq:nonextend}). 
\end{remark}

It is known that if $X \setminus B_{(1+\epsilon)R}(x) \neq \emptyset$ for some $\epsilon>0$, then
\begin{equation}\label{eq:00}
\int_{B_R(x)}|f|^2 \dist \meas \le C(K, N, R, \epsilon) \int_{B_R(x)}\Gamma(f)\dist \meas\qquad \forall f \in H^{1, 2}_0(B_R(x), \dist, \meas),
\end{equation}
which implies $\lambda_1^D(B_R(x))>0$ (c.f. \cite[(4.5)]{Cheeger}).

Let us give a simple corollary of Theorem~\ref{thm:conv stab}.
We first state the following:
\begin{lemma}\label{lem:app}
For any $f \in H^{1, 2}(B_R(x), \dist, \meas)$, $r\in (0,R)$, there exist 
$f_i \in H^{1, 2}(B_r(x_i), \dist_i, \meas_i) $ such that $f_i\vert_{B_r(x_i)}$ converge strongly to $f\vert_{B_r(x)}$ on $B_r(x)$ in $H^{1, 2}$.
\end{lemma}
\begin{proof} By using cut-off functions, with no loss of generality we can assume there exists $g \in H^{1, 2}(X, \dist, \meas)$ such that 
$g\vert_{B_r(x)} = f\vert_{B_r(x)}$.
Then, applying Theorem~\ref{thm:Mosco}(b) to $g$ together with Theorem~\ref{thm:cont_reco} completes 
the proof.
\end{proof}

\begin{remark}
In Lemma~\ref{lem:app}, the assumption that $r$ is strictly smaller than $R$ is needed. See Remark~\ref{rem:example} above.
In connection with this problem, the authors do not know whether there exists a generic condition to satisfy the Mosco convergence of the following energy;
\begin{equation}\label{eq:local energy3}
\tilde{\Ch}_{B_R(x)}(f):=
\begin{cases}\frac{1}{2}\int_{B_R(x)}\Gamma(f)\dist \meas &\text{if $f|_{B_R(x)}\in H^{1, 2}(B_R(x), \dist, \meas)$;}\\
+\infty &\text{otherwise}
\end{cases}
\end{equation}
with respect to the mGH-convergence.
%
%
\end{remark}

The next approximation result,  was known in the noncollapsed metric cone setting \cite{Ding2}
and in the compact Ricci limit setting \cite{Honda3}.
Our result extends these to general $\RCD^*(K, N)$-spaces, without extra strong assumptions.

\begin{corollary}[Harmonic approximation/replacement]\label{cor:harm harm}
Let $f$ be a harmonic function on $B_R(x)$. Then,  for any $r\in (0,R)$ there exist harmonic functions 
$f_{i}$ on $B_r(x_i)$ such that $f_i$ strongly converge to $f\vert_{B_r(x)}$ in $H^{1, 2}$ on $B_r(x)$.
Moreover, for any $g_i \in H^{1, 2}(B_r(x_i), \dist_i, \meas_i)$ converging strongly to $f\vert_{B_r(x)}$ in $H^{1, 2}$ on $B_r(x)$ 
for some $r\le R$, the harmonic replacements $\hat{g}_i$ of $g_i$ converge strongly to $f\vert_{B_s(x)}$ in $H^{1, 2}$ on $B_s(x)$
for all $s \in (0, r]$ with
$$
H^{1, 2}_0(B_s(x), \dist, \meas) =\hat{H}^{1, 2}_0(B_s(x), \dist, \meas).
$$  
\end{corollary}
\begin{proof} 
Take $\bar s\in (r, R)$ with $H^{1, 2}_0(B_{\bar s}(x), \dist, \meas) =\hat{H}^{1, 2}_0(B_{\bar s}(x), \dist, \meas)$.
Note that it is not restrictive to assume that $\lambda_1^D(B_{\bar s}(x)) >0$, since all harmonic functions are
constant otherwise.

Then by Lemma~\ref{lem:app} there exist $h_i \in H^{1, 2}(B_{\bar s}(x_i), \dist_i, \meas_i)$ such that $h_i$ converge strongly to 
$f\vert_{B_{\bar s}(x)}$ in $H^{1, 2}$. Then, thanks to Theorem \ref{thm:conv stab}, the harmonic replacements $\hat{h}_i$ of $h_i$
converge strongly to $f\vert_{B_{\bar s}(x)}$ in $H^{1, 2}$ on $B_{\bar s}(x)$.
Thus, letting $f_i:=\hat{h}_i\vert_{B_r(x_i)}$, completes the proof of the first part.
Moreover, the final statement about radii $s\in (0,r]$ can also be easily checked by this argument.
\end{proof}

\begin{remark}
One might wonder about a global version of Corollary~\ref{cor:harm harm}:
\begin{itemize}
\item[] for any harmonic function $f$ on $X$ (which means that $f\vert_{B_R(x)}$ is harmonic for any $R>0$), are there a subsequence $i(j)$ 
and harmonic functions $f_{i(j)}$ on $X_{i(j)}$ such that $f_{i(j)}$ $L^2$-strongly converge to $f$ on $B_R(x)$ for any $R>0$?
\end{itemize}
The following simple example shows that the global version, as stated, does not hold.
First we check that if $g$ is harmonic on $[0,+\infty)$ (with respect to $\mathcal{L}^1$), then $g$ is constant. Indeed,
clearly $g$ has to be affine and the condition $g'(0)\phi(0)=-\int_0^{+\infty}g'(t)\phi '(t)\dist t=0$
for any $\phi \in C^{\infty}_c(\mathbb{R})$, yields $g'\equiv 0$.

Let us consider the pointed mGH-convergent family
\begin{equation}\label{eq:example2}
([0, +\infty),\dist_{eucl}, s, \mathcal{L}^1) \stackrel{mGH}{\to} (\mathbb{R}, \dist_{eucl},0, \mathcal{L}^1)\qquad
s\uparrow\infty.
\end{equation}
Then, any nonconstant harmonic function $f$ on $\mathbb{R}$ has no $L^2_{\mathrm{loc}}$-strong approximation by global harmonic functions with respect to the convergence (\ref{eq:example2}).
\end{remark}

\begin{remark}\label{rem:excess}
In the proof of the splitting theorem for Ricci limit spaces proven in \cite{CheegerColding}, the 
\textit{harmonic replacement} of the distance function played an important role, as follows.
\begin{itemize}
\item[] Let $(M_i, \dist_i, x_i,\frac{\mathrm{vol}}{\mathrm{vol}\,B_1(m_i)})$ be a pointed mGH-convergent sequence of $n$-dimensional Riemannian manifolds to a Ricci limit space $(X, \dist,x, \meas)$ with $\mathrm{Ric}_{M_i} \ge -\delta_i$,  $\delta_i \to 0$.
If $X$ contains a line $\gamma: \mathbb{R} \to X$ with $\gamma (0)=x$, then for any $\epsilon>0$ and any $R>0$  there exists 
$L_1>1$ such that for any $L \ge L_1$ and any sequence $x_i^L \in M_i$ converging to $\gamma (L) \in X$,
there exist harmonic functions $h_i$ on $B_R(x^L_i)$ with 
\begin{equation}\label{eq:harm rep ric}
\limsup_{i \to \infty}\|e_i^L-h_i\|_{H^{1, 2}(B_R(x^L_i))}<\epsilon,
\end{equation}
where $e_i^L(z):=\dist_i(x_i^L, x_i)-\dist_i (x_i^L, z)$. Moreover, we can take $h_i$ as the harmonic replacements of $e_i^L$
in $B_R(x_i^L)$.
\end{itemize}

Even though in the splitting theorem on $\RCD^*(0, N)$-spaces established in \cite{Gigli2} the harmonic replacement were not needed
(because a priori these spaces do not arise from an approximation), recently harmonic replacement have been considered in $\RCD$-setting in \cite{HanMondino}. It is worth pointing out that in the proof of harmonic replacement as above, the sharp Laplacian comparison theorem and the 
maximum principle play a key role.

We are now in a position to prove, via Corollary~\ref{cor:harm harm}, harmonic replacement for \textit{most} $R>0$ using neither 
the ``almost'' nonnegative Ricci curvature condition,  nor Laplacian comparison theorem and maximum principle. 
To see how via Corollary \ref{cor:harm harm} can be applied, we assume that the limit space 
$(X, \dist, x,\meas)$ of $(X_i, \dist_i, x_i,\meas_i)$ satisfies;
$$
(X, \dist, \meas) := (\mathbb{R}, \dist_{eucl}, \mathcal{H}^1) \times (Y, \dist_Y, \nu)
$$
for some $\RCD^*(\hat{K}, \hat{N})$-space $ (Y, \dist_Y, \nu)$ and we
denote by $\gamma$ a canonical line in $X$, i.e. $\gamma(t)=(t, y)$ with $\gamma (0)=x$ for a fixed $y \in Y$.
Note that the Busemann function $b_{\gamma}$ of $\gamma$ is equal to the projection to the $\mathbb{R}$-factor, 
i.e. $b_{\gamma}(t, z)=t$. In particular $b_{\gamma}$ is harmonic on $X$ with $\Gamma( b_{\gamma})=1$ $\meas$-a.e. in $X$.

Now, fix $R>0$ with 
$$
H^{1, 2}_0(B_R(x), \dist, \meas)=\hat{H}^{1, 2}_0(B_R(x), \dist, \meas)
$$
and, for any $L \ge 1$, take a convergent sequence $x_i^L \in X_i$ to $\gamma (L) \in X$.
Then it is easy to see that there exists a subsequence $i(j)$ such that $e_{i(j)}^j$ converge 
uniformly to $b_{\gamma}$ on each bounded subset of $X$ (in particular it is an $L^2_{\mathrm{loc}}$-strongly convergent
sequence), where 
$e_{i(j)}^j(z):=\dist_{i(j)}(x_{i(j)}^j, x_{i(j)})-\dist_{i(j)}(x_{i(j)}^j, z)$.
Since $\Gamma_{i(j)}(e_{i(j)}^j)=1$ $\meas_{i(j)}$-a.e., we have
$$
\lim_{j \to \infty}\int_{B_R(x_{i(j)})}\Gamma_{i(j)}( e_{i(j)}^j)\dist \meas_{i(j)}=\lim_{j \to \infty}\meas_{i(j)}(B_R(x_{i(j)}))=\meas (B_R(x))
=\int_{B_R(x)}\Gamma(b_{\gamma})\dist \meas.
$$
Thus, Corollary~\ref{cor:harm harm} can be applied in this case and it follows that
the harmonic replacements $\hat{e}_{i(j)}$ of $e_{i(j)}^j$ on $B_R(x_{i(j)})$ satisfy 
$$
\|e_{i(j)}^j-\hat{e}_{i(j)}\|_{H^{1, 2}(B_R(x_{i(j)}))} \to 0.
$$\end{remark}

\begin{proposition}[Local convergence for metric cones]\label{prop:cone sobolev}
Let us consider the metric cone
$$
(C(Z), \dist_{C(Z)}, r^2\meas_Z)
$$
over a compact $\RCD^*(N-2,N-1)$-space $(Z, \dist_Z, \meas_Z)$
(by \cite{Ketterer15a} $(C(Z), \dist_{C(Z)}, r^2\meas_Z)$ is an $\RCD^*(0, N)$-space).
Then 
\begin{equation}\label{eq:cone}
H^{1, 2}_0(B_R(p), \dist_{C(Z)}, r^2\meas_Z) =\hat{H}^{1, 2}_0(B_R(p), \dist_{C(Z)}, r^2\meas_Z)\qquad\forall R>0,
\end{equation}
where $p=(0, *)$ is the pole of $C(Z)=([0, \infty) \times Z)/(\{0\} \times Z)$.
\end{proposition}
\begin{proof}
For any $\epsilon \in (0, 1)$, let us consider the $1$-Lipschitz map $\Phi_{\epsilon}: C(Z) \to C(Z)$ defined by $\Phi_{\epsilon}((t, z)):=((1-\epsilon)t, z)$.
By Lemma~\ref{lem:ch}, for any $f \in \hat{H}^{1, 2}(B_R(p), \dist_{C(Z)}, r^2\meas_Z)$ there exists $f_{\epsilon} \in \mathrm{LIP}_c(B_{(1-\epsilon)^{-1}R}(p), \dist_{C(Z)})$ such that $\|f-f_{\epsilon}\|_{H^{1, 2}}<\epsilon$.
Then, since it is easy to check that the support of $g_{\epsilon}:=f_{\epsilon} \circ \Phi_{\epsilon}$ is included in $B_R(p)$ and that $\|g_{\epsilon} - f\|_{H^{1, 2}} \to 0$ as $\epsilon \downarrow 0$, we obtain that $f \in H^{1, 2}_0(B_R(p), \dist_{C(Z)}, r^2\meas_Z)$, which completes the proof of (\ref{eq:cone}). 
\end{proof}

We end this section by giving the following stability result, already known in the case when the sequence consists of noncollapsed Riemannian manifolds with almost nonnegative Ricci curvature in \cite{CheegerColding}.

\begin{theorem}\label{thm:metric cone distance}
Assume that the limit space $(X, \dist, x,\meas)$ is a $\RCD^*(0,N)$ space satisfying
$$r \mapsto \frac{\meas (B_r(x))}{r^N}\qquad\text{is constant in $(0,+\infty)$}.$$
Then the solutions of the Dirichlet problem, $\Delta_{x_i,R} \hat{f}_i=2N$ with $\dist (x_i, \cdot )^2-\hat{f}_i \in H^{1, 2}_0(B_R(x_i), \dist_i, \meas_i)$, satisfy
\begin{equation}\label{eq:cone rep}
\lim_{i\to\infty}\|\dist ( x_i, \cdot)^2 -\hat{f}_i\|_{H^{1, 2}}= 0\qquad\forall R>0.
\end{equation}
\end{theorem}
\begin{proof}
Recall that from \cite{DePhilippisGigli} we can represent
$$
(X, x, \dist, \meas) = (C(Z), p, \dist_{C(Z)}, r^2\meas_Z)
$$
for some $\RCD^*(N-2,N-1)$-space $(Z, \dist_Z, \meas_Z)$ and that 
$\dist_{C(Z)} (p, \cdot)^2 \vert_{B_R(p)} \in \mathcal{D}(\Delta, B_R(p))$ with 
$\Delta_{p,R} \dist_{C(Z)} (p, \cdot)^2=2N$.
Since $\dist_i(x_i, \cdot)^2\vert_{B_R(x_i)}$ converge strongly to $\dist_{C(Z)}(p, \cdot)^2$ in $H^{1, 2}$ on $B_R(p)$, 
applying Theorem~\ref{thm:conv stab} with Proposition~\ref{prop:cone sobolev} yields (\ref{eq:cone rep}).
\end{proof}

\end{document}